\documentclass[11pt,oneside]{amsart}
\usepackage{amsfonts,amsmath,amsthm}
\usepackage{amssymb,epsfig}
\usepackage{pgf,tikz}
\usepackage{fancybox}
\usepackage[T1]{fontenc}
\usepackage{stmaryrd}
\usepackage{graphicx}
\usepackage{eso-pic}
\usepackage{geometry}
\usepackage{tikz-cd}
\usepackage[unicode=true]{hyperref}
\hypersetup{colorlinks = true}
\hypersetup{
    colorlinks = true,
    linkcolor = {blue},
    citecolor={blue}
}
\usepackage[above,below,verbose,section]{placeins}
\usepackage{mathrsfs}
\usepackage{mathpazo}
\usepackage{lastpage}
\usepackage{enumitem}

\makeatother

\usepackage[colorinlistoftodos,prependcaption]{todonotes}
\presetkeys%
{todonotes}%
{inline,backgroundcolor=white}{}
\tikzset{/tikz/notestyleraw/.append style={text=black}}

\newtheorem{thm}{Theorem}[section]

\newtheorem{lem}[thm]{Lemma}
\newtheorem{defn}[thm]{Definition}
\newtheorem{prop}[thm]{Proposition}
\newtheorem{cor}[thm]{Corollary}

\newtheorem{rmk}[thm]{Remark}

\newcommand{\be}{\begin{eqnarray}}
\newcommand{\ee}{\end{eqnarray}}
\newcommand{\ben}{\begin{eqnarray*}}
\newcommand{\een}{\end{eqnarray*}}
\newcommand{\beq}{\begin{equation}}
\newcommand{\eeq}{\end{equation}}
\newcommand{\beal}{\begin{aligned}}
\newcommand{\enal}{\end{aligned}}

\newcommand{\eps}{\varepsilon}

\newcommand{\T}{\mathbb{T}}
\newcommand{\R}{\mathbb{R}}

\newcommand{\N}{\mathbb{N}}

\newcommand{\Z}{\mathbb{Z}}

\newcommand{\om}{\omega}

\newcommand{\cM}{\mathcal{M}}

\newcommand{\cP}{\mathcal{P}}

\newcommand{\cL}{\mathcal{L}}

\newcommand{\wh}{\widehat }
\newcommand{\wt}{\widetilde }

\title{ Statistical regularity and linear response of Mather Measures for Tonelli Lagrangian systems}

\author{Alfonso Sorrentino}
\address{(Alfonso Sorrentino) Dipartimento di Matematica, Universit\'e degli  Studi di Roma ``Tor Vergata'', Via della ricerca scientifica 1, 00133 Rome, Italy}
\email{sorrentino@mat.uniroma2.it}

\author{Jianlu Zhang}
\address{
(Jianlu Zhang) State Key Laboratory of Mathematical Sciences, Academy of Mathematics and systems science, Chinese Academy of Sciences, Beijing 100190, China}
\email{jellychung1987@gmail.com}

\author{Siyao Zhu}
\address{
(Siyao Zhu) State Key Laboratory of Mathematical Sciences, Academy of Mathematics and systems science, Chinese Academy of Sciences, Beijing 100190, China}
\email{zsydtc@amss.ac.cn}

\thanks{{\it Statements and Declarations:} The authors declare no competing interests and no data were generated or analyzed in this study.}
\date{}
\begin{document}

\maketitle

%
%
%

\begin{abstract}
We study the statistical regularity of Mather measures associated with $C^1$ perturbations of a Tonelli Lagrangian. When the unperturbed Mather measure is supported on a quasi-periodic torus with a Diophantine frequency, we establish H\"older continuity of the perturbed Mather measure with respect to the perturbation parameter. The H\"older exponent is shown to depend explicitly on the Diophantine index of the frequency. We also discuss the possibility of achieving Lipschitz regularity using KAM theory.
\end{abstract}

\bigskip


\section{Introduction}

The Aubry-Mather theory provides a comprehensive description of the minimizing measures for Tonelli Lagrangian systems. A central concept is that of \textit{statistical stability}, which concerns how these Mather measures, as elements of the space of probability measures, change under perturbations of the system. This notion is intimately linked to quantitative estimates of the impact of perturbations on the system's dynamics.

\medskip
In this article, we consider two common types of perturbations:
\begin{enumerate}
    \item \textit{Ma\~n\'e's perturbation} \cite{Mane96}: \(L_\eps(x,v) = L(x,v) + \eps f(x)\) with \(0<\eps\ll 1\) and \(f\in C(\T^n,\R)\).
    \item \textit{Cohomological perturbation} \cite{Mather91}: \(L_c(x,v) = L(x,v) - \langle c,v \rangle\) with \(c\in H^1(\T^n,\R)\) and \(|c|\ll 1\).
\end{enumerate}

In suitable topologies, \(L_\eps \to L\) as \(\eps\to 0\) and \(L_c \to L\) as \(c\to {0}\). Our goal is to investigate the regularity of the associated Mather measures \(\mu_\eps\) and \(\mu_c\) as functions of the parameters \(\eps\) and \(c\). To make this precise, we equip the space of Borel probability measures on the tangent bundle $T\T^n$,  that we denote \(\cP(T\T^n)\), with the so-called \(1\)-Wasserstein distance (see \cite{Mane96} for an alternative choice of topology).

We say the map \(\eps \mapsto \mu_\eps\) (respectively, \(c \mapsto \mu_c\)) is \textit{statistically stable} if it is continuous at \(0\) (or \({0}\)). The notion of \textit{statistical regularity} refines this by providing a modulus of continuity. Naturally, one aims for stronger forms of regularity, such as H\"older or Lipschitz continuity, or even differentiability (the so called {\it linear response}). These are more challenging to obtain and typically require more detailed dynamical information about the perturbed system.

\begin{rmk}
For uniformly hyperbolic systems, this program is well-developed, with results on Lipschitz regularity and, in some cases, differentiability with respect to perturbations (see \cite{Bal, Gal}). In such systems, it is also known that generic Mather measures are supported on a unique hyperbolic periodic orbit \cite{Con}. However, Tonelli Lagrangian systems are generally not uniformly hyperbolic, and much less is known about the regularity of their Mather measures. The famous Ma\~n\'e conjecture posits that for generic Lagrangians, the Mather measure is ergodic and supported on a unique hyperbolic periodic orbit. While significant progress has been made from a perturbative viewpoint \cite{FR1, FR2, Con, Con2}, this conjecture remains largely open. 
\end{rmk}

In this paper, we focus on the two classes of perturbations mentioned above. Under the key assumption that the unperturbed Mather measure of \(L\) is supported on a quasi-periodic torus with a non-resonant frequency, we establish H\"older continuity for the perturbed Mather measures. Since such a quasi-periodic torus is not hyperbolic, this result provides a first step toward understanding the statistical stability of non-hyperbolic systems. \\

\medskip

\noindent {\bf Summary of our main results:}

\smallskip

\begin{itemize}
\item[(1)]  {\bf (Ma\~n\'e's Lagrangian case)} Suppose $L_\eps(x,v)=\frac12|v-V_\eps(x)|^2$ with a  vector field $V_\eps:\T^n\rightarrow \R^n$, which depends Lipschitz on $(x,\eps)\in\T^n\times[0,1]$ and $V_0\equiv \om\in\R^n$, where $\omega$ satisfies one of the these two conditions:
\begin{itemize}
\item  $\omega \in \mathcal{D}_{\sigma,\tau}$ is a finite Diophantine vector (see \eqref{def:Diophantine in Rn} for a precise definition);
		\item $\omega = (\omega_1,\ldots,\omega_n)$ with each component $\omega_i$ being algebraic\footnote{A real number is called \emph{algebraic} if it is a root of a non-zero polynomial with integer coefficients.}, for $i = 1,\ldots,n$.
\end{itemize}

\smallskip
Then: 

\medskip
\begin{itemize}
\item[(i)]  {\rm (Upper bound)}
	The Mather measure set $\wt{\mathfrak{M}}_\eps$ (see Definition \ref{Mathermeasure}) is H\"older continuous w.r.t. $\eps$, namely 
	\[
	W_1(\widetilde{\mathfrak{M}}_0, \widetilde{\mathfrak{M}}_\eps)
	\le C\,\delta^{\,l}
	\]
	for some suitable constants $l=l(n,\om)\in(0,1)$ and $C = C(n,\omega)>0$ (see Remark~\ref{rmkA2} for the expression of $l$ for each cases of $\om$). Here $W_1(\cdot,\cdot)$ denotes the  $1-$Wasserstein distance (see \eqref{eq:Wasserstein} for a precise definition).
	
	\smallskip
	\noindent (See Theorem~\ref{Thm1} for its proof)
	\medskip
	
\item[(ii)] {\rm (Lower bound)} Suppose $n=2$ and $\om\in\mathcal D_{\sigma.\tau}$. For any $0 < r < \tau$, there exists a decreasing sequence $\{\eps_j\}_{j\in\Z_+}$ tending to $0$ as $j\to+\infty$, such that 
	\[
	W_1\bigl(\widetilde{\mathfrak{M}}_0,\,\widetilde{\mathfrak{M}}_{\eps_j}\bigr)
	\ge c\,\eps_j^{1/(r+1)}
	\]
	for some constant  $c>0$ independent of $j\in\Z_+$.
	
	\medskip
	\noindent (See Theorem~\ref{Thm5} for its proof)\\
 \end{itemize}
 \item[(2)]{\bf (Cohomological perturbation case)} Suppose $L_c(x,v)=L(x,v)-\langle c,v\rangle$ is a family of Tonelli Lagrangians parametrized by $c\in H_1(\T^n,\R)$ with $|c|\leq 1$. Associated to $c=0$, the Mather measure set $\wt{\mathfrak M}(0)$ is a singleton and the unique Mather measure is supported on a KAM torus with  frequency $\om\in\mathcal D_{\sigma,\tau}$. Then there exists \(C=C(n,\omega,H_0)>0\) such that: 
 \[
		W_1\big(\widetilde{\mathfrak{M}}(c), \widetilde{\mathfrak{M}}(0)\big)
		\le C\,|c|^{\frac{1}{\,n+\tau+1\,}} \qquad \forall\; c\in H_1(\T^n,\R), \; |c|\leq 1.
	\]
	
	\medskip
	\noindent (See Theorem \ref{Thm4} for its proof)\\
	
 \item[(3)]{\bf (Ma\~n\'e's perturbation case)} 	Suppose $L_\eps(x,v)=L(v)-\eps f(x)$ is a family of Tonelli Lagrangians  parametrized  by $\eps\in [0,1)$ and assume that $\wt{\mathfrak M}(0)$, the Mather measure set associated to $\eps=0$, consists of a single measure supported on a KAM torus with frequency $\om\in\mathcal D_{\sigma,\tau}$.
 \smallskip
 \begin{itemize}
 \item[(i)] Then, there exists a modulus of continuity \(\Theta : \mathbb{R}^+ \to \mathbb{R}^+\), with \(\lim_{r\to 0^+}\Theta(r)=0\), such that:
	\[
	W_1 \bigl(\widetilde{\mathfrak{M}}(\varepsilon), \widetilde{\mathfrak{M}}(0)\bigr) \;\leq\; \Theta(\varepsilon) \qquad \forall\; \eps\in [0,1).
	\]
	
		\medskip
	\noindent (See Theorem~\ref{Thm3} for its proof)
	\smallskip

 \item[(ii)] {(Linear response)} There exists a sequence $c_\eps\in H_1(\T^n,\R)$ tending to ${\bf 0}$ as $\eps\rightarrow 0$, such that the associated 
	$\widetilde{\mathfrak{M}}(c_\varepsilon)$ admits a linear response, {\it i.e.}, 
	\[
\widetilde{\mathfrak{M}}_1 := 
\lim_{\varepsilon \to 0} 
\frac{\widetilde{\mathfrak{M}}(c_\varepsilon)-\widetilde{\mathfrak{M}}(0)}{\varepsilon}
\]
exists in the sense of weak convergence of measures. Moreover,  for any 
	$g \in C_c^\infty(T\mathbb{T}^n,\mathbb{R})$ we have 
	\begin{align*}
	 \int g\, d\widetilde{\mathfrak{M}}_1
	&= \langle \partial_1 g(x,\omega), \psi_1(x,\varepsilon) \rangle \\
	&\quad +\sum_{k\in\Z^n\backslash \{{\bf 0}\}} \frac{\langle \partial_2 g(x,\omega), \partial_{pp}^{-1}L(0)k\rangle}{\langle k,\omega\rangle} f_k \, e^{i\langle k,x\rangle}\\
	&\quad +\langle \partial_2 g(x,\omega), \partial_{pp}^{-1}L(0)c_1\rangle,
	\end{align*}
where $c_1=\lim_{\eps\rightarrow 0_+}c_\eps/\eps$, $\psi_1(x,\varepsilon)$ is defined in \eqref{eq:psi-expansion} and $f_k$ denotes the Fourier coefficient of $f$ for $k\in\mathbb{Z}^n$.

\medskip
\noindent (See  Theorem~\ref{Thm:KAM-W1} for its proof)
\end{itemize}
\end{itemize}

\begin{rmk}
The results in  item (1) are  used in a fundamental way in the proof of the results in items (2) and (3). We also remark that item (2) partially answers an open problem posed by Walter Craig in \cite{Bolotin03openproblem}, which sought to connect the regularity of Mather measures with the arithmetic properties of the derivative of Mather's \(\alpha\)-function (see Remark \ref{rmk:DiophantineImproved} in Section \ref{Sec: generalcase}). Moreover, since KAM theory guarantees that  quasi-periodic tori are prevalent, this makes our results relevant for understanding perturbative dynamics in a broader context.
\end{rmk}

\medskip

\noindent{\bf Structure of the paper.} This article is organized as follows:
\begin{itemize}
\item  In Section \ref{sec:preliminary} we gather the necessary preliminary material that will be used throughout the rest of the article.
    \item In Section~\ref{Sec:Mane L}, we begin with the special case of Ma\~n\'e's Lagrangian of the  form \eqref{eq:ManeL}. This simpler setting allows us to obtain H\"older continuity of the Mather measures without heavy dynamical machinery. We establish an upper bound for the convergence rate in Theorem~\ref{Thm1} and a lower bound for the two-dimensional case in Theorem~\ref{Thm5}. 
    In  Remark \ref{rmkA2} and Remark \ref{rmk:2}, the relation between the H\"older exponent and the Diophantine index is analyzed. Even under these favorable assumptions, a gap remains between the upper and lower bounds. For \(n>2\), a general lower bound is unattainable, but we discuss special arithmetic conditions that yield a rough estimate in subsection~\ref{Sec:3.3}.
    We emphasize that in our setting the perturbation is only Lipschitz, which requires an additional discussion of Aubry-Mather theory in the case of low regularity.

    \item In Section~\ref{Sec: generalcase}, we use weak KAM theory to generalize the results to arbitrary Tonelli perturbations of a quadratic Lagrangian whose Mather measure is supported on a quasi-periodic torus. By transforming the perturbed system into a Ma\~n\'e-type Lagrangian that shares the same Mather measures, the conclusions of Section~\ref{Sec:Mane L} can be applied. We present regularity results for cohomological perturbations in Theorem~\ref{Thm4} and for Ma\~{n}\'{e} potential perturbations in Theorem~\ref{Thm3}.

    \item In Section~\ref{Sec:Linearresponse}, we discuss the possibility of Lipschitz continuity and differentiability (linear response) using KAM theory. The main result is stated in Theorem~\ref{Thm:KAM-W1}.
\end{itemize}

\medskip

Finally, we briefly comment on the methodology. The results in Sections \ref{Sec:Mane L} and \ref{Sec: generalcase} rely on a quantitative estimate for the convergence rate of Birkhoff averages for quasi-periodic systems. This estimate, borrowed from \cite{ZhangCPDE}, is fundamentally linked to the Diophantine nature of the frequency via the Denjoy-Koksma inequality \cite{Herman79}. The discussion of linear response in Section~\ref{Sec:Linearresponse} draws on techniques from \cite{YanETDS2014} for the regularity of viscosity solutions with respect to cohomological perturbations, and from \cite{CarandSorrentino2021} for the smoothness of Mather's \(\beta\)-function.

\medskip

\noindent{\bf Acknowledgements.} JZ acknowleges the support of the National Key R\&D Program of China (No. 2022YFA1007500) and the National NSF of China (No. 12231010, 12571207).
AS acknowledges the support of the Italian Ministry of University and Research's PRIN 2022 grant ``\textit{Stability in Hamiltonian dynamics and beyond}'' and of the Department of Excellence grant MatMod@TOV (2023-27), awarded to the Department of Mathematics at University of Rome Tor Vergata. AS is member of the INdAM research group GNAMPA and the UMI group DinAmicI.
SZ acknowleges the the Postdoctoral Fellowship Program of CPSF under Grant Number GZB20250720.


\section{Preliminaries}\label{sec:preliminary}

\subsection{Notation and setting}
Throughout this paper, \( \mathbb{T}^n := \mathbb{R}^n / \mathbb{Z}^n \) denotes the flat \( n \)-dimensional torus, equipped with the standard Euclidean norm \(|\cdot|\). The tangent and cotangent bundles over \( \mathbb{T}^n \) are denoted by \( T\mathbb{T}^n \) and \( T^*\mathbb{T}^n \), with coordinates \( (x,v) \) and \( (x,p) \), respectively. For both, \( \pi: T\mathbb{T}^n \text{ (resp. } T^*\mathbb{T}^n\text{)} \to \mathbb{T}^n \) is the canonical projection. The Euclidean norm is naturally inherited by the fibers.

\medskip
\( C(\mathbb{T}^n,\mathbb{R}) \) denotes the space of continuous functions with the sup norm \(\|\cdot\|_\infty\). The spaces of \(k\)-times continuously differentiable and Lipschitz continuous functions are denoted by \( C^k(\mathbb{T}^n,\mathbb{R}) \) and \( C^{0,1}(\mathbb{T}^n,\mathbb{R}) \). We denote by $C_c(T\T^n, \R)$ the space of continuous functions on $T\T^n$ with compact support.

\medskip
A vector $\omega \in \mathbb{R}^n$ is said to be \emph{non-resonant} if
\[
\langle k, \omega\rangle \neq 0 \quad \text{for all } k \in \mathbb{Z}^n \setminus \{{0}\}.
\]
More strongly, $\omega \in \mathbb{R}^n$ is \emph{Diophantine with exponent} $\tau \ge n-1$ if there exists a constant $\sigma> 0$ such that
\begin{equation}\label{def:Diophantine in Rn}
|\langle \omega, k \rangle| \;\geq\; \sigma |k|^{-\tau}, \quad \forall\, k \in \mathbb{Z}^n \setminus \{{0}\}.
\end{equation}
Denote by $\mathcal D_{\sigma,\tau}\subset\R^n$ the set of such vectors.

\begin{rmk} 
For an irrational number $\nu \in \mathbb{R}$, the Diophantine exponent corresponding to it can be characterized as
\be\label{def:Diophantine in R}
\tau(\nu) \;=\; \sup \left\{ \gamma \geq 0 : 
\liminf_{m \to +\infty} m^\gamma \,\| m \nu\|_{\mathbb{Z}} = 0 \right\}
\ee
where $\| \cdot \|_{\mathbb{Z}}$ denotes the distance to the nearest integer.

\end{rmk}

\subsection{Viscosity solutions}
Consider a Hamiltonian \( H: T^*\mathbb{T}^n \to \mathbb{R} \) that is continuous, convex, and superlinear in the momentum variable \( p \). Its Lagrangian conjugate \( L: T\mathbb{T}^n \to \mathbb{R} \) is defined by the Fenchel transform
\[
L(x,v) := \sup_{p \in T_x^*\mathbb{T}^n} \bigl\{ \langle p, v \rangle - H(x,p) \bigr\}.
\]
For such a Hamiltonian, there exists a unique constant \(c(H) \in \mathbb{R}\), called the \emph{ergodic constant} or \emph{Ma\~n\'e critical value}, such that the Hamilton-Jacobi equation
\begin{equation}\label{H-J}
H(x, du) = c(H), \quad x \in \mathbb{T}^n
\end{equation}
admits a viscosity solution \cite{Fathiweakkam, Mather91}.

\begin{defn}
    A function \( u \in C(\mathbb{T}^n, \mathbb{R}) \) is a \emph{viscosity subsolution} of \eqref{H-J} if, for every \( x_0 \in \mathbb{T}^n \) and every \( C^1 \) function \( \varphi \) such that \( u - \varphi \) has a local maximum at \( x_0 \), we have \( H(x_0, d\varphi(x_0)) \le c(H) \). It is a \emph{viscosity supersolution} if for every local minimum of \( u - \varphi \) at \( x_0 \), we have \( H(x_0, d\varphi(x_0)) \ge c(H) \). A function that is both a subsolution and a supersolution is a \emph{viscosity solution}.
\end{defn}

In the following, all solutions, subsolutions, and supersolutions are understood with respect to this critical value $c(H)$ unless otherwise stated. 

A key concept for non-smooth subsolutions is that of a \emph{reachable gradient} \cite{Cannarsa_Sinestrari_book}.
\begin{defn}
    Let \( u \) be a viscosity subsolution of \eqref{H-J}. A vector \( p \in T_x^*\mathbb{T}^n \) is a \emph{reachable gradient} of \(u\) at \(x\) if there exists a sequence \(\{x_k\} \subset \mathbb{T}^n \setminus \{x\}\) such that \(u\) is differentiable at each \(x_k\) and \(\lim_{k\to\infty} (x_k, du(x_k)) = (x,p)\). The set of all reachable gradients of \(u\) at \(x\) is denoted by \(D^* u(x)\).
\end{defn}

\subsection{Aubry-Mather theory for low-regularity Lagrangians}
We now assume \(L\) is only continuous. 
For every \(t>0\), we define the \emph{minimal action function} \(h_t : \mathbb{T}^n \times \mathbb{T}^n \to \mathbb{R}\) by
\[
h_t(x, y) := \inf_{\gamma} \int_{-t}^{0} L\big(\gamma(s), \dot \gamma(s)\big) + c(H)\, ds,
\]
where the infimum is taken over all absolutely continuous curves \(\gamma : [-t,0] \to M\) satisfying \(\gamma(-t) = x\) and \(\gamma(0) = y\). The \emph{Peierls barrier} is defined as
\[
h^{\infty}(x, y) := \liminf_{t \to +\infty} h_t(x, y);
\]
it turns out to be well posed for any $x\in\T^n$ (see \cite{DaviniFathi2016}), the function $h^\infty(x,\cdot)$ is a viscosity solution of \eqref{H-J}.
The \emph{projected Aubry set} is defined as
\[
\mathcal{A} := \{ x \in \mathbb{T}^n : h^{\infty}(x,x) = 0 \}.
\]

\begin{lem}\cite[Theorem 4.8]{Davini06}\label{lem:staticcurve}
    For any \( y \in \mathcal{A} \), there exists an absolutely continuous curve \( \gamma : \mathbb{R} \to \mathcal{A} \) with \( \gamma(0) = y \) such that for all \(a\le b\),
    \begin{equation}\label{eq:defn-stat}
    \int_a^b L(\gamma(s), \dot{\gamma}(s)) + c(H)\, ds + h^\infty(\gamma(b), \gamma(a)) = 0.
    \end{equation}
    Such a curve is called \emph{static}.
\end{lem}

For any critical subsolution \(u\) of \eqref{H-J}, define
\begin{equation}\label{eq:Lu}
\mathcal{L}(u) := \{ (x,v) \in T\mathbb{T}^n \mid \partial^+ u(x,v) = L(x,v) + c(H) \},
\end{equation}
where \(\partial^+ u(x,v) = \sup \{ p(v) : p \in D^* u(x) \}\) denotes {\it Clarke upper directional derivative}.

\medskip
The \emph{Aubry set} is then defined as the intersection
\[
\widetilde{\mathcal{A}} := \bigcap_u \mathcal{L}(u),
\]
where the intersection is taken over all subsolutions $u$. The following lemma shows that the definition of the Aubry set $\widetilde{\mathcal{A}}$ is well-defined.

\begin{lem}\label{lem:welldefinedofA}
	The canonical projection $\pi: \widetilde{\mathcal{A}} \to \mathcal{A}$ is surjective. 
\end{lem}


\begin{proof}
This result has already been established in Theorem B.14 of \cite{DaviniFathi2016}, and we only sketch the main ideas of the proof here (see also \cite{Davini06,Fathi05cv} for related conclusions).
For each critical subsolution \( u \) and any \( x \in \mathcal{A} \), by Lemma \ref{lem:staticcurve}, there exists a static curve 
\( \gamma : \mathbb{R} \to \mathbb{T}^n \) with \( \gamma(0) = x \). Along this curve, $u$ is calibrated, that is,
\[
u(\gamma(b)) - u(\gamma(a)) 
= \int_a^b \big( L(\gamma(s), \dot\gamma(s)) + c(H) \big)\, ds, 
\quad \forall a<b.
\]
By combining this property with the definition of static curve, one deduces that, for almost every differentiability point $s$ of $\gamma$, the equality in \eqref{eq:Lu} is satisfied, that is,
\[
(\gamma(s), \dot \gamma(s)) \in \cL(u).
\]
By continuity of $\gamma$ and the compactness of $\mathcal{L}(u)$, its trajectory  entirely lies in the projection $\pi(\mathcal{L}(u))$.

At a non-differentiable point $s_0$, there exists a sequence of differentiable points $s_n\to s_0$ such that
\[
\partial^+u(\gamma(s_n), \dot\gamma(s_n))
= L(\gamma(s_n), \dot\gamma(s_n)) + c(H).
\]
By the continuity of $L$ and of $u$ along $\gamma$, we deduce that 
$(\gamma(s_0), \lim_{n\to\infty}\dot\gamma(s_n))$ satisfies \eqref{eq:Lu}. Finally, note that $\gamma$ does not depend on the choice of $u$, which shows that each $\mathcal{L}(u)$ contains a common nonempty compact subset independent of $u$. This proves the surjectivity of $\pi$. 
\end{proof}

\begin{defn}[Mather measure]\label{Mathermeasure}
    A probability measure \(\mu\in\cP(T\mathbb{T}^n)\) is \emph{closed} if \(\int |v|\,d\mu < \infty\) and \(\int d\phi(x)\cdot v \, d\mu(x,v) = 0\) for all \(\phi\in C^1(\mathbb{T}^n,\mathbb{R})\). Let \(\mathfrak{P}_0\) be the set of closed measures. The Ma\~n\'e critical value satisfies
    \begin{equation}\label{mane critical value}
    \min_{\mu \in \mathfrak{P}_0} \int L(x,v) \, d\mu(x,v) = -c(H).
    \end{equation}
    Any minimizer \(\mu\) in \eqref{mane critical value} is called a \emph{Mather measure} (associated with \(L\)). Denote by \(\widetilde{\mathfrak{M}}\) the set of all Mather measures.
\end{defn}

The set of \emph{projected Mather measures} is defined by
\[
\mathfrak{M} := \pi_{\#} \widetilde{\mathfrak{M}} 
= \{\, \pi_{\#} \widetilde{\mu} \mid \widetilde{\mu} \in \widetilde{\mathfrak{M}} \,\},
\]
where $\pi_{\#}$ denotes the pushforward operator associated with the canonical projection 
$\pi : T\mathbb{T}^n \to \mathbb{T}^n$. Equivalently, for any $f \in C(\mathbb{T}^n, \mathbb{R})$,
\[
\int_{\mathbb{T}^n} f(x) \, d(\pi_{\#} \widetilde{\mu})(x) 
= \int_{T\mathbb{T}^n} (f \circ \pi)(x,v) \, d\widetilde{\mu}(x,v).
\]
The \emph{Mather set} and the \emph{projected Mather set} are respectively defined as follows:
\begin{align} \label{def:matherset}
	\widetilde{\mathcal{M}} 
	:= \overline{\bigcup_{\widetilde{\mu} \in \widetilde{\mathfrak{M}}} 
		\mathrm{supp}(\widetilde{\mu})}, 
	\qquad 
	\mathcal{M} := \pi(\widetilde{\mathcal{M}}) \subset \mathbb{T}^n.
\end{align}

\begin{lem}\cite[Corollary~10.3]{Fathi05cv}\label{lem:mathersetisinAubry}
    \(\widetilde{\mathcal{M}}\subset\widetilde{\mathcal{A}}\).
\end{lem}

\begin{rmk}
If \(H\) (and hence \(L\)) is \(C^2\) and strictly convex, it is called a {Tonelli Hamiltonian/Lagrangian}. In this case, Mather measures are invariant under the Euler-Lagrange flow \cite{Mather91}
\end{rmk}

\medskip

Before concluding this section, we  recall the \(1\)-Wasserstein distance (also known as the Kantorovich-Rubinstein distance). For \(\mu,\nu\in\mathcal{P}_1(T\mathbb{T}^n)\), it is defined as
\begin{equation}\label{eq:Wasserstein}
W_1(\mu,\nu) = \inf_{\pi\in\Pi(\mu,\nu)} \int_{\mathbb{T}^n\times\mathbb{T}^n} d(x,y) \, d\pi(x,y),
\end{equation}
where \(\Pi(\mu,\nu)\) is the set of couplings. A useful dual formulation is
\begin{equation}\label{eq:KRdistance}
W_1(\mu,\nu) = \sup_{\psi\in\mathrm{Lip}_1(\mathbb{T}^n)} \left\{ \int \psi \, d\mu - \int \psi \, d\nu \right\},
\end{equation}
where \(\mathrm{Lip}_1\) denotes 1-Lipschitz functions. See \cite{VillaniOT} for details.

\begin{rmk}\label{rmk:optimaltransport}
This distance is related to the so-called Monge's problem.
Given two probability measures $\mu$ and $\nu$ on metric spaces $X$ and $Y$, respectively, and a cost function $c:X\times Y\to\mathbb{R}$, Monge's problem consists in finding a measurable map $T:X\to Y$ that pushes $\mu$ forward to $\nu$ (i.e. $T_\#\mu=\nu$) and minimizes the total  cost
\begin{equation}\label{eq:Monge}
	\inf_{T_\#\mu=\nu}\int_X c(x,T(x))\,d\mu(x).
\end{equation}
A transport map ({\it i.e.}, a map achieving the infimum) may not exist in general, which motivates Kantorovich's relaxation or Kantorovich's problem:
let \(\Pi(\mu, \nu)\) be the set of probability measures \(\pi\) on \(X \times Y\) with marginals \(\mu\) and \(\nu\). The relaxed problem reads  
\begin{equation}\label{eq:Kantorovich}
	\inf_{\pi \in \Pi(\mu, \nu)} \int_{X \times Y} c(x, y) \, d\pi(x, y).
\end{equation}
Wasserstein distance arises naturally when the cost function is chosen as a power of the metric. More specifically,
for $p\in[1,\infty)$, let $\cP_p(X)$ be the space of Borel probability measures on $(X,d)$ having finite $p$-th moment.
The \(p\)-Wasserstein distance between \(\mu, \nu \in \mathcal{P}_p(X)\) is given by:  
\begin{equation}\label{eq:Wasserstein}
	W_p(\mu, \nu) := \left( \inf_{\pi \in \Pi(\mu, \nu)} \int_{X \times X} d(x, y)^p \, d\pi(x, y) \right)^{1/p}.
\end{equation}
\end{rmk}

\bigskip

\section{Statistical regularity of the Mather measures for Ma\~n\'e's Lagrangians}\label{Sec:Mane L}

We now consider a family of Ma\~n\'e's Lagrangians \( L_\delta : T\mathbb{T}^n  \to \mathbb{R} \), defined by
\begin{equation}\label{eq:ManeL}
L_\delta(x, v) := \frac{1}{2} \left| v - V_\delta(x) \right|^2,
\end{equation}
where \( \delta \geq 0 \) is a small parameter and \( V_\delta : \mathbb{T}^n \to \mathbb{R}^n \) is a family of Lipschitz vector fields satisfying the following assumptions:
\begin{enumerate}[label=(A\arabic*),ref=(A\arabic*)]
	\item\label{ass:A1} $V_0(x) = \omega \in \mathbb{R}^n$ is a non-resonant vector.
	\item\label{ass:A2} $\omega$ satisfies one of the following two conditions:
	\begin{enumerate}[label=(\roman*)]
		\item $\omega \in \mathcal{D}_{\sigma,\tau}$ is a finite Diophantine vector;
		\item $\omega = (\omega_1,\ldots,\omega_n)$ is a vector whose components $\omega_i$ are all algebraic for $i = 1,\ldots,n$.
	\end{enumerate}
	\item\label{ass:A3}
	For \(0\le\delta\ll1\), we have
	\[
	\sup_{x\in\mathbb{T}^n} \lvert V_\delta(x)-\omega\rvert \le \delta.
	\]
\end{enumerate}

\noindent
It is worth noting that, whether (i) or (ii) in Assumption~\ref{ass:A2} holds, the corresponding vectors always form a prevalent subset of~$\mathbb{R}^n$.
Indeed, for case~(i), if $\tau>n-1$, the Diophantine set $\mathcal{D}_{\sigma,\tau}$ has asymptotically full Lebesgue measure as $\sigma\to0$.
For case~(ii), Lemma 2.10 in \cite{ZhangCPDE} shows that the set of non-resonant vectors whose components are algebraic also has full measure. We also point out that for  $n=2$, thanks  to \textit{Roth's theorem} (see \cite{Cassels1957}), any non-resonant vector $\omega$ with all algebraic components necessarily satisfies a finite Diophantine condition, i.e., $\omega\in \mathcal{D}_{\sigma,\tau}$ for suitable $\sigma\in[1,2)$; however, for $n\geq 3$, we do not know whether (ii) implies (i).

\begin{lem}\label{invarianceflow} 
	Let \(\mathfrak{M}_\delta\) denote the set of projected Mather measures associated with \(L_\delta\).  
	For any \(\mu_\delta \in \mathfrak{M}_\delta\), one has
	\[
	(\varphi_\delta^t)_\# \mu_\delta = \mu_\delta,
	\qquad \forall\, t \in \mathbb{R},
	\]
	where \(\varphi_\delta^t : \mathbb{T}^n \to \mathbb{T}^n\) denotes the flow generated by the ODE
	\begin{equation}\label{ODEfor t}
		\dot{x} = V_\delta(x), \qquad t \in \mathbb{R}.
	\end{equation}
	In other words, every \(\mu_\delta \in \mathfrak{M}_\delta\) is invariant under the flow \(\varphi_\delta^t\). Consequently, \(\widetilde{\mathfrak{M}}_\delta\) is nonempty. 
\end{lem}
\begin{proof}

\medskip
\noindent
\textbf{Step 1: Existence of Mather measures.} 
Since $L_\delta$ is of Mañé type, its critical value is $c(H_\delta) = 0$, i.e., $\inf_{\tilde{\mu} \in \mathfrak{P}_0} \int_{T\mathbb{T}^n} L_\delta \, d\tilde{\mu} = 0$. For any initial point \(x \in \mathbb{T}^n\), there exists an integral curve \(\gamma : [0,T] \to \mathbb{T}^n\) solving equation \eqref{ODEfor t} with \(\gamma(0) = x\).
Using this curve \(\gamma\), we define a probability measure \(\tilde{\mu}_\gamma\) on the tangent bundle \(T\mathbb{T}^n\) by
\[
\int_{T\mathbb{T}^n} \psi(x,v) \, d\tilde{\mu}_\gamma(x,v) := \frac{1}{T} \int_0^T \psi(\gamma(t), \dot{\gamma}(t)) \, dt,
\]
for any bounded measurable function \(\psi : T\mathbb{T}^n \to \mathbb{R}\).

We have
\[
\int_{T\mathbb{T}^n} |v| \, d\tilde{\mu}_\gamma(x,v) = \frac{1}{T} \int_0^T |\dot{\gamma}(t)| \, dt = \frac{\ell_g(\gamma)}{T} < +\infty,
\]
where $\ell_g(\gamma)$ denotes the Riemannian length of $\gamma$, which is finite since $\gamma$ is absolutely continuous on the compact interval $[0,T]$.
Moreover, for any $f \in C^1(\mathbb{T}^n,\mathbb{R})$,
\[
\int_{T\mathbb{T}^n} df_x(v) \, d\tilde{\mu}_\gamma(x,v) = \frac{1}{T} \int_0^T \frac{d}{dt} f(\gamma(t)) \, dt = \frac{f(\gamma(T)) - f(\gamma(0))}{T}.
\]
Taking a sequence \(T_k \to \infty\), by the compactness of the space of probability measures and the Banach-Alaoglu theorem, there exists a subsequence \(\tilde{\mu}_{\gamma_{k}}\) converging in the weak\(^*\) topology to a probability measure \(\tilde{\mu}_\infty\).
Moreover, since $\gamma(t)$ is an integral curve of \eqref{ODEfor t}, we have
\[
\int_{T\mathbb{T}^n} L_\delta(x,v) \, d\tilde{\mu}_\infty(x,v) = 0,
\]
which implies \(\tilde{\mu}_\infty\) is a minimizing Mather measure associated with \(L_\delta\).
Therefore, the Mather measure set is nonempty, i.e., $\widetilde{\mathfrak{M}}_\delta \neq \emptyset$.

\medskip
\noindent
 \textbf{Step 2: Graph property of the Mather set.} 
Since $u\equiv 0$ is a smooth solution of the Mañé-type Hamilton-Jacobi equation, the Aubry set $\widetilde{\mathcal{A}}_\delta$ satisfies
\[
\widetilde{\mathcal{A}}_\delta \subset \mathcal{L}(0) = \{(x,V_\delta(x)) \mid x \in \mathcal{A}_\delta\}.
\]
By Lemma \ref{lem:mathersetisinAubry}, the Mather set is contained in the Aubry set. Together with the inclusion above, we conclude that  
\[
\widetilde{\mathcal{M}}_\delta = \{(x,V_\delta(x)) \mid x \in \mathcal{M}_\delta \subset \mathcal{A}_\delta\}.
\]

\medskip
\noindent
\textbf{Step 3: Invariance of projected Mather measures.} 
For any \(\mu_\delta \in \mathfrak{M}_\delta\), the graph property of \(\widetilde{\mathcal{M}}_\delta\) implies that  
\begin{equation}\label{eq:mathermeasuregraph}
	\int_{T\mathbb{T}^n} \Phi(x, v) \, d\tilde{\mu}_\delta(x, v) 
	= \int_{\mathbb{T}^n} \Phi\big(x, V_\delta(x)\big) \, d\mu_\delta(x), 
	\quad \forall \Phi \in C_c(T\mathbb{T}^n, \mathbb{R}).
\end{equation}

Since \(\tilde{\mu}_\delta\) is a closed measure, for any \(\phi \in C^1(\mathbb{T}^n, \mathbb{R})\) we have
\begin{equation}\label{eq:closed_measure}
	\int_{T\mathbb{T}^n} d\phi_x(v) \, d\tilde{\mu}_\delta(x, v) = 0.
\end{equation}
By taking \(\Phi(x,v) = d\phi_x(v)\) in~\eqref{eq:mathermeasuregraph} and combining with~\eqref{eq:closed_measure}, it follows that
\begin{equation}\label{eq:invariance_integral}
	\int_{\mathbb{T}^n} d\phi_x \big( V_\delta(x) \big) \, d\mu_\delta(x) = 0,
	\quad \forall \phi \in C^1(\mathbb{T}^n, \mathbb{R}).
\end{equation}

Let \(\phi \in C^1(\mathbb{T}^n,\mathbb{R})\). We compute
\[
\begin{aligned}
	\int_{\mathbb{T}^n} \phi \, d \big( (\varphi_\delta^t)_\# \mu_\delta - \mu_\delta \big) 
	&= \int_{\mathbb{T}^n} \big( \phi(\varphi_\delta^t(x)) - \phi(x) \big) \, d\mu_\delta(x) \\
	&= \int_{\mathbb{T}^n} \int_0^t \frac{d}{ds} \phi(\varphi_\delta^s(x)) \, ds \, d\mu_\delta(x) \\
	&= \int_{\mathbb{T}^n} \int_0^t D\phi(\varphi_\delta^s(x)) \cdot V_\delta(\varphi_\delta^s(x)) \, ds \, d\mu_\delta(x),
\end{aligned}
\]
where we have applied the fundamental theorem of calculus, the chain rule, and the definition of the flow \(\varphi_\delta^t\).

Set \( \phi_s(x) := \phi(\varphi_\delta^s(x)) \) for each fixed \( s \in [0,t] \).  
By~\eqref{eq:invariance_integral} applied to \(\phi_s\), we have
\[
\int_{\mathbb{T}^n} D\phi_s(x) \cdot V_\delta(x) \, d\mu_\delta(x) = 0.
\]
Thus,
\[
\int_{\mathbb{T}^n} \phi \, d \big( (\varphi_\delta^t)_\# \mu_\delta - \mu_\delta \big) 
= \int_0^t \int_{\mathbb{T}^n} D\phi_s(x) \cdot V_\delta(x) \, d\mu_\delta(x) \, ds = 0.
\]
Since this holds for every \(\phi \in C^1(\mathbb{T}^n,\mathbb{R})\), we conclude that \(\mu_\delta\) is invariant under the flow \(\varphi_\delta^t\).

Moreover, by the density of \(C^1(\mathbb{T}^n,\mathbb{R})\) in \(\mathrm{Lip}(\mathbb{T}^n,\mathbb{R})\) with respect to the uniform norm and the dominated convergence theorem, the same conclusion extends to all Lipschitz test functions:
\[
\int_{\mathbb{T}^n} \phi \, d \big( (\varphi_\delta^t)_\# \mu_\delta - \mu_\delta \big) = 0,
\quad \forall \phi \in \mathrm{Lip}(\mathbb{T}^n,\mathbb{R}).
\]
The proof of Lemma \ref{invarianceflow} is now complete. 
\end{proof}

Due to Assumptions~\ref{ass:A1}--\ref{ass:A3}, \(\widetilde{\mathfrak{M}}_\delta\) (resp. \(\mathfrak{M}_\delta\)) may contain multiple elements, while \(\widetilde{\mathfrak{M}}_0\) (resp. \(\mathfrak{M}_0\)) is a singleton. As we are concerned with the Wasserstein distance between any Mather measure of \(L_\delta\) and that of \(L_0\), we shall, without risk of confusion, write \(\widetilde{\mathfrak{M}}_\delta\) (resp.~\(\mathfrak{M}_\delta\)) to indicate an arbitrary choice of measure from the respective set.

\subsection{Upper bound}
In this subsection, we establish an upper bound for the Wasserstein distance between Mather measures under perturbations. Our analysis relies on a quantitative ergodic lemma, which will serve as a key technical ingredient:
\begin{lem}\label{lem1}
	Let \( F \in C^{0, \alpha}(\mathbb{T}^n) \) with \( n \geq 2 \).  
	\begin{enumerate}
	\item[(i)] [Theorem~3 in \cite{Klein2021ETDS}]  Let $\omega = (\omega_1, \ldots, \omega_n) \in \mathcal{D}_{\sigma,\tau}$. 
	There exists a constant \(C = C\bigl(n, \sigma,\tau, \|F\|_{C^{0,\alpha}}\bigr)\) such that
	\[
	\left| \frac{1}{T} \int_0^T F(x + \omega t) \, dt 
	- \int_{\mathbb{T}^n} F({x}) \, d {x} \right| 
	\leq C \, T^{-\frac{\alpha}{\tau + n}}.
	\]
	\item[(ii)] [Proposition~2.12 in \cite{ZhangCPDE}] Let \( \omega = (\omega_1, \ldots, \omega_n) \in \mathbb{R}^n \) be non-resonant with \( \omega_i \) algebraic for \( i = 1, \ldots, n \). For every \( 0 < \varepsilon \ll 1 \), there exists a constant \( C = C(n, \omega, \|F\|_{C^{0, \alpha}}, \varepsilon) \) such that:
	\[
	\left| \frac{1}{T} \int_0^T F(x + \omega t) \, dt - \int_{\mathbb{T}^n} F({x}) \, d{x} \right| \leq C_\varepsilon T^{-\frac{\alpha}{\alpha + n - 1+ \varepsilon} }.
	\]
	\item[(iii)] [Proposition~2.12 in \cite{ZhangCPDE}] If \( n = 2 \), \( \omega = (\omega_1, \omega_2) \), and the Diophantine index \( \tau_{\omega} = 1 \), then there exists a constant \( C = C(\tau, \|F\|_{C^{0, \alpha}}) \) such that:
	\[
	\left| \frac{1}{T} \int_0^T F(x + \omega t) \, dt - \int_{\mathbb{T}^n} F({x}) \, d{x} \right| \leq C T^{-\frac{\alpha}{\alpha + 1}}.
	\]
	\end{enumerate}
\end{lem}

\begin{thm}[H\"older continuity]\label{Thm1}
	Assume that Assumptions \ref{ass:A1} --\ref{ass:A3} hold.
	Then there exist constants $l\in(0,1)$ and $C = C(n,\omega)$ such that, for all $\delta>0$ sufficiently small,
	\[
	W_1(\mathfrak{M}_0, \mathfrak{M}_\delta)
	\le C\,\delta^{\,l},
	\qquad
	W_1(\widetilde{\mathfrak{M}}_0, \widetilde{\mathfrak{M}}_\delta)
	\le C\,\delta^{\,l}.
	\]
	The exponent $l$ depends only on $n$ and $\omega$ (see Remark~\ref{rmkA2} for a discussion on its possible values).
\end{thm}
\begin{proof}
We present the proof under the first case of Assumption~\ref{ass:A2} together with~\ref{ass:A3}; the other case can be treated similarly. By the Kantorovich--Rubinstein duality formula (see Remark \ref{rmk:optimaltransport}), we have
\[
W_1(\widetilde{\mathfrak{M}}_0, \widetilde{\mathfrak{M}}_\delta) = \sup_{\|f\|_{\mathrm{Lip}} \leq 1} \left| \int_{T\mathbb{T}^n} f\, d\widetilde{\mathfrak{M}}_0 - \int_{T\mathbb{T}^n} f\, d\widetilde{\mathfrak{M}}_\delta \right|.
\]
By the graph property of Mather measures, we can rewrite the right-hand side as
\[
W_1(\widetilde{\mathfrak{M}}_0, \widetilde{\mathfrak{M}}_\delta) = \sup_{\|f\|_{\mathrm{Lip}} \leq 1} \left| \int_{\mathbb{T}^n} f(x, V_0(x))\, d\mathfrak{M}_0 - \int_{\mathbb{T}^n} f(x, V_\delta(x))\, d\mathfrak{M}_\delta \right|.
\]
To decompose the difference, we add and subtract the quantity $\int_{\mathbb{T}^n} f(x, V_0(x))\, d\mathfrak{M}_\delta$, obtaining
\begin{align}\label{eq:inequalitydecomposed}
	W_1(\widetilde{\mathfrak{M}}_0, \widetilde{\mathfrak{M}}_\delta)
	&= \sup_{\|f\|_{\mathrm{Lip}}\leq 1} \Bigg| \int_{\mathbb{T}^n} f(x, V_0(x))\, d\mathfrak{M}_0
	- \int_{\mathbb{T}^n} f(x, V_0(x))\, d\mathfrak{M}_\delta \notag \\
	&\qquad\quad + \int_{\mathbb{T}^n} f(x, V_0(x))\, d\mathfrak{M}_\delta
	- \int_{\mathbb{T}^n} f(x, V_\delta(x))\, d\mathfrak{M}_\delta \Bigg| \notag \\
	&\leq \sup_{\|f\|_{\mathrm{Lip}} \leq 1} \Bigg[ \left| \int_{\mathbb{T}^n} f(x, V_0(x))\, d(\mathfrak{M}_0 - \mathfrak{M}_\delta) \right| \notag \\
	&\qquad\qquad\qquad + \int_{\mathbb{T}^n} \left| f(x, V_0(x)) - f(x, V_\delta(x)) \right| \, d\mathfrak{M}_\delta \Bigg]. 	
\end{align}
By Assumption~\ref{ass:A3}, the second term in \eqref{eq:inequalitydecomposed} is bounded above by \(\delta\). Hence, we obtain the following inequality:
\begin{equation}\label{estimation of Mather measure}
W_1\big(\widetilde{\mathfrak{M}}_0, \widetilde{\mathfrak{M}}_\delta\big) \leq  W_1\big(\mathfrak{M}_0, \mathfrak{M}_\delta\big) +\delta .
\end{equation}

By Lemma \ref{invarianceflow}, we have the following equalities:
\[
\mathfrak{M}_0=(\varphi_0^t)_{\#} \mathfrak{M}_0, \quad \mathfrak{M}_\delta=(\varphi_\delta^t)_{\#} \mathfrak{M}_\delta
\]
for every $t\in \mathbb{R}$. Then by the Kantorovich--Rubinstein duality formula, for any $T>0$,
\begin{align}
	W_1(\mathfrak{M}_0, \mathfrak{M}_\delta)
	&=\sup_{\|f\|_{\mathrm{Lip}}\leq 1} \left| \int_{\mathbb{T}^n} f \, d\mathfrak{M}_0 - \int_{\mathbb{T}^n} f \, d\mathfrak{M}_\delta \right| \nonumber \\
	&\leq \sup_{\|f\|_{\mathrm{Lip}} \leq 1} \Bigg| \int_{\mathbb{T}^n} f \, d\mathfrak{M}_0 -\frac{1}{T} \int_0^T \int_{\mathbb{T}^n}f \circ \varphi_0^t \, d\mathfrak{M}_\delta \, dt \nonumber\\
	&\qquad\qquad\qquad +\frac{1}{T}\int_0^T \int_{\mathbb{T}^n} \big( f \circ \varphi_0^t - f \circ \varphi_\delta^t \big) \, d\mathfrak{M}_\delta \, dt \Bigg| \label{E5}
\end{align}
Note that every probability measure on $\mathbb{T}^n$ can be approximated in the weak topology by a convex combination of Dirac measures. More precisely, for any $\varepsilon>0$, there exist points $x_1,\dots,x_m\in \mathbb{T}^n$ and weights $\lambda_1,\dots,\lambda_m\geq 0$ with $\sum_{i=1}^m \lambda_i=1$ such that  
\[
\left|\int_{\mathbb{T}^n}\phi \, d\mathfrak{M}_\delta - \sum_{i=1}^m \lambda_i \phi(x_i)\right| \leq \varepsilon,
\quad \forall \phi \in C(\mathbb{T}^n, \mathbb{R}).
\]  
Since $\omega$ is non-resonant, the function $t \mapsto f(x+\omega t)$ is quasi periodic\footnote{We call a function $\phi:\R\rightarrow\R$ quasi periodic, if there exists a function $\psi:\T^m\rightarrow\R$ and a non-resonant vector $\nu\in\R^m$, such that $\phi(x)=\psi(x\cdot \nu)$. 
}. Using Lemma~\ref{lem1}, we obtain the estimate  
\begin{align}\label{E6}
	&\left| \int_{\mathbb{T}^n} f \, d\mathfrak{M}_0 
	- \frac{1}{T}\int_0^T \int_{\mathbb{T}^n} f(x+\omega t) \, d\mathfrak{M}_\delta \, dt \right| \nonumber\\
	&\quad \leq \left| \int_{\mathbb{T}^n} f \, d\mathfrak{M}_0 
	- \frac{1}{T}\int_0^T \sum_{i=1}^m \lambda_i f(x_i+\omega t)\, dt \right| + \varepsilon \nonumber\\
	&\quad \leq C T^{-\beta} + \varepsilon,
\end{align}
where $\beta =\frac{\alpha}{\tau+n} \in(0,1)$ is a constant. Since $\varepsilon>0$ is arbitrary, we conclude
\[
\left| \int_{\mathbb{T}^n} f \, d\mathfrak{M}_0 - \frac{1}{T}\int_0^T \int_{\mathbb{T}^n} f(x+\omega t) \, d\mathfrak{M}_\delta \, dt \right| \le C T^{-\beta}.
\]
By equation \eqref{ODEfor t}, we have
\[
\varphi_\delta^t(x)=x+\int_0^t V_\delta(\varphi_\delta^s(x))\,ds.
\]
Together with $\|f\|_{\mathrm{Lip}}\leq 1$, this implies 
\begin{align*}
	\frac{1}{T} \int_0^T \int_{\mathbb{T}^n} | f\circ \varphi_0^t(x) - f\circ \varphi_\delta^t(x) | \, d\mathfrak{M}_\delta(x) \, dt 
	&\leq \frac{1}{T} \int_0^T \int_{\mathbb{T}^n} \int_0^t |\omega - V_\delta(\varphi_\delta^s(x))| \, ds \, d\mathfrak{M}_\delta(x) \, dt.
\end{align*}
Using Assumption~\ref{ass:A3}, we have 
\begin{align}\label{E7}
	\frac{1}{T} \int_0^T \int_{\mathbb{T}^n} | f\circ \varphi_0^t -f\circ \varphi_\delta^t | \, d\mathfrak{M}_\delta \, dt 
	&\leq \frac{1}{T} \int_0^T \int_{\mathbb{T}^n} \int_0^t \delta \, ds \, d\mathfrak{M}_\delta \, dt \leq \frac{T\delta}{2}.
\end{align}
Substituting the estimates from \eqref{E6} and \eqref{E7} into inequality \eqref{E5}, we derive the following bound:
\begin{align}\label{E8}
	W_1(\mathfrak{M}_0, \mathfrak{M}_\delta) \leq C_1 T^{-\beta} + C_2 T \delta,
\end{align}
where \( C_1, C_2 > 0 \) are constants.

Since \eqref{E8} holds for any \( T > 0 \), we can minimize the function 
\[
g(T) := C_1 T^{-\beta} + C_2 T \delta.
\]
The minimum point of \( g(T) \) is given by 
\[
T_{\text{min}} = \tilde{C} \delta^{-\frac{1}{\beta+1}}
\]
for some constant \( \tilde{C} > 0 \). Evaluating \( g(T_{\text{min}}) \), we have
\[
g_{\text{min}} \leq C_3 \delta^{\frac{\beta}{\beta+1}}.
\]
Substituting this into \eqref{E8}, we deduce that
\[
W_1(\mathfrak{M}_0, \mathfrak{M}_\delta) \leq C \delta^{\frac{\beta}{\beta+1}}.
\]
Substituting the above inequality into inequality \eqref{estimation of Mather measure}, we finally obtain
\[
W_1\big(\widetilde{\mathfrak{M}}_0, \widetilde{\mathfrak{M}}_\delta\big) \leq C \delta^{\,l},
\quad \text{with } l = \frac{\beta}{\beta+1} \in (0,1).
\]
Here, $l$ is a constant depending on $n$ and $\omega$. This completes the proof. 
\end{proof}

\begin{rmk}\label{rmkA2}
\begin{enumerate}[label=(\roman*)]
		\item The H\"older exponent for the $W_1$-distance of the Mather measure \emph{strongly depends} on the convergence rate provided by Lemma~\ref{lem1}. When $\omega \in \mathcal D_{\sigma,\tau}$, we have $\beta = \alpha/(\tau+n)$ and the exponent $l$ is given by
		\[
		l = \frac{\alpha}{\alpha + \tau + n}.
		\]
		For Lipschitz observables ($\alpha=1$), this simplifies to \( l = 1/(1+\tau+n) \). In particular, if $n=2$ and $\tau=1$, this yields $l=1/4$.
		\item If $\omega$ is non-resonant and each component $\omega_i$ is algebraic, then by Schmidt's subspace theorem, $\omega$ is Diophantine with some exponent $\tau$. For such $\omega$, the exponent $l$ can be made arbitrarily close to \( 1/(n+1) \) by taking $\varepsilon$ small in Lemma~\ref{lem1} (ii), namely
		the exponent $l$ in Theorem \ref{Thm1} improves to
		\be\label{eq:l2}
		l = \frac{1}{n+\varepsilon+1},
		\ee
		where $\varepsilon>0$ is arbitrarily small. 
		\item In fact, it is sufficient to estimate \(|V_\delta-V_0|\) restricted to the projected Mather set \(\mathcal{M}_\delta\) in \eqref{eq:inequalitydecomposed} and \eqref{E7}, and the proof remains valid.
\end{enumerate}
\end{rmk}

\medskip

\subsection{Lower bound in dimension 2}
In this subsection, we obtain a lower bound for the $W_1$-distance between the Mather measures in the case $n=2$. The reason we specifically choose $n=2$ is that Dirichlet's approximation theorem guarantees that for any irrational $\nu\in\R$, there exist infinitely many pairs \((q,p) \in \mathbb{Z}^2\) with $q>0$ such that
\begin{equation}\label{eq:Dirichlet}
\left| \nu - \frac{p}{q} \right| \le \frac 1{q^{2}}.
\end{equation}
On the other hand, the Diophantine exponent of $(\nu,1)\in\R^2$ is at least $1$, which matches the approximation rate in \eqref{eq:Dirichlet}. This yields a sharp lower bound for the Wasserstein distance when $\nu$ has Diophantine exponent $1$.
\begin{rmk}
For $n>3$, there is a  gap between the minimal Diophantine index and the exponent in the Dirichlet's approximation (see \cite{ChengActa2011}), therefore it is not possible to get an analogue of Diophantine approximation to the frequency. Nonetheless, in subsection \ref{Sec:3.3} we still provide a possible lower bound under special arithmetic conditions.
\end{rmk}

\begin{defn}[Physical measure]
	Let $\varphi_V^t : \mathbb{T}^n \to \mathbb{T}^n$ be the continuous-time flow generated by a Lipschitz vector field $V : \mathbb{T}^n \to T\mathbb{T}^n$. For any $x \in \mathbb{T}^n$ and $T > 0$, define the empirical measure
	$\mu_T^x := \frac{1}{T} \int_0^T \delta_{\varphi_V^t(x)} \, dt$,
	where $\delta_y$ denotes the Dirac measure at $y \in \mathbb{T}^n$.  
	An invariant Borel probability measure $\mu$ for $\varphi_V^t$ is called a \emph{physical measure} if the set
	\[
	\mathcal{B}(\mu) := \left\{ x \in \mathbb{T}^n : \lim_{T \to \infty} \mu_T^x = \mu \text{ in the weak* topology} \right\}
	\]
	has positive Lebesgue measure. The set $\mathcal{B}(\mu)$ is called the \emph{basin of attraction} of $\mu$.
\end{defn}

\begin{lem}\label{Thm2}
	Let $n=2$ and $\omega=(1,\nu)$, where $\nu$ is Diophantine with exponent $1<\tau(\nu)<\infty$.
	Then, for any $0<r<\tau(\nu)$, there exist a sequence $\{\delta_j\}_{j\in\mathbb{N}}$
	with $\delta_j>0$, $\delta_j\to0$ as $j\to\infty$ and a corresponding sequence of
	vector field  $V_{\delta_j}$ satifying 
	\[
	\sup_{x\in\mathbb{T}^2}|V_0-V_{\delta_j}(x)|=\delta_j,
	\]
such that 
	\begin{equation}\label{E9}
		W_1\bigl(\mathfrak{M}_0,\mathfrak{M}_{\delta_j}\bigr)
		\ge c\,\delta_j^{1/(r+1)}
		\qquad \text{for all } j\in\mathbb{N}.
	\end{equation}
	for some constant $c>0$.
\end{lem}
\begin{proof}
Without loss of generality, we assume that $\nu<1$. For any $r<\tau(\nu)$, the definition of $\tau(\nu)$ implies the existence of infinitely many integers $k_n \in \mathbb{N}$ and $p_n \in \mathbb{Z}$ such that 
\[
\left|\nu-\frac{p_n}{k_n}\right|\leq \frac{1}{k_n^{r+1}}.
\]
We set $V_{\delta_n}(x):=(1, p_n/k_n)$ for all $x\in \mathbb{T}^2$. Then $\delta_n=|V_{\delta_n}-V_0| \leq 1/k_n^{r+1} \to 0$ as $n \to \infty$.

For each rational $p_n/k_n$, the corresponding flow on $\mathbb T^2$ admits a periodic orbit. Introduce the global transversal section
\[
\Sigma := \{(x,y)\in\mathbb T^2:\ x=0\}.
\]
The associated Poincaré map $P_n:\Sigma\to\Sigma$ is a rigid rotation given by
\[
P_n(y)=y+\frac{p_n}{k_n} \bmod 1.
\]
Consequently, the periodic orbit of the flow intersects $\Sigma$ at exactly $k_n$ distinct points. In the fundamental domain $(0,1)\times (0,1)$, the closed orbit decomposes into $k_n$ parallel line segments:
\[
l_i := \bigl\{(x,y)\in(0,1)^2:\ y=\tfrac{p_n}{k_n}x+\tfrac{i}{k_n} \bmod 1 \bigr\}, \quad i=0,1,\dots,k_n-1.
\]
We endow each $l_i$ with normalized arc-length measure and define the probability measure
\[
\mathfrak M_{\delta_n} := \frac{1}{k_n} \sum_{i=0}^{k_n-1} \mathfrak{m}|_{l_i},
\]
supported on the periodic orbit. To obtain a lower bound for $W_1(\mathfrak M_0,\mathfrak M_{\delta_n})$, consider the function 
\[
\psi(x) := \min_{0\le i\le k_n-1} \operatorname{dist}(x,l_i),
\]
where $\operatorname{dist}(x,l_i)$ denotes the Euclidean distance to $l_i$ in $\mathbb{T}^2$. Since $\psi$ is the minimum of finitely many distance functions, it is $1$-Lipschitz. Therefore, by Kantorovich--Rubinstein duality,
\[
W_1(\mathfrak M_0,\mathfrak M_{\delta_n}) \ge \int_{\mathbb{T}^2}\psi\,d\mathfrak M_0 - \int_{\mathbb{T}^2}\psi\,d\mathfrak M_{\delta_n} = \int_{\mathbb{T}^2}\psi\,d\mathfrak M_0,
\]
because $\psi$ vanishes on $\operatorname{supp}(\mathfrak M_{\delta_n})$.

The distance between consecutive lines $l_i$ and $l_{i+1}$ is $d_n = 1/(k_n\sqrt{1+(p_n/k_n)^2})$. A direct computation gives
\[
\int_{\mathbb{T}^2}\psi\,d\mathfrak M_0 = k_n \int_0^{d_n} \min(t, d_n-t)\,dt = \frac{d_n^2}{4} = \frac{1}{4k_n^2\bigl(1+(p_n/k_n)^2\bigr)}.
\]
Since $p_n/k_n$ is bounded, there exists $c>0$ such that
\[
W_1(\mathfrak M_0,\mathfrak M_{\delta_n}) \ge \frac{c}{k_n}.
\]
Because $\delta_n \le k_n^{-(r+1)}$, we obtain $k_n \le \delta_n^{-1/(r+1)}$, and thus
\[
W_1(\mathfrak M_0,\mathfrak M_{\delta_n}) \ge c\,\delta_n^{1/(r+1)}.
\]

To ensure that $\mathfrak M_{\delta_n}$ is indeed a physical measure for a smooth perturbation, we construct a smooth vector field $V_{\delta_n}$ satisfying the given bound and having the periodic orbit as an attractor (this construction is inspired by Proposition 18 in \cite{Sorrentino21}).
For each strip between $l_i$ and $l_{i+1}$, define the midline
\[
m_i := \{(x,y)\in(0,1)^2 : y = \tfrac{p_n}{k_n} x + \tfrac{i+1/2}{k_n} \bmod 1\}.
\]
Perturb the vector field in the vertical direction:
\[
V_{\delta_n}(x,y) := \bigl(1, \tfrac{p_n}{k_n} + \delta\, g(x,y)\bigr),
\]
where $\delta>0$ is small, $g:(0,1)\times (0,1) \to\mathbb{R}$ is a smooth function, positive above $m_i$, negative below $m_i$, and vanishing on each $l_i$ and $m_i$, with appropriate monotonicity.  More specifically, for each $i$:
\begin{itemize}
	\item $g(x,y)<0$ for $y \in (\tfrac{p_n}{k_n} x + \tfrac{i}{k_n}, \tfrac{p_n}{k_n} x + \tfrac{i+1/2}{k_n})$, monotone increasing there,
	\item $g(x,y)>0$ for $y \in (\tfrac{p_n}{k_n} x + \tfrac{i+1/2}{k_n}, \tfrac{p_n}{k_n} x + \tfrac{i+1}{k_n})$, monotone decreasing there,
	\item $g(x, \tfrac{p_n}{k_n} x + \tfrac{i}{k_n}) = g(x, \tfrac{p_n}{k_n} x + \tfrac{i+1/2}{k_n}) = g(x, \tfrac{p_n}{k_n} x + \tfrac{i+1}{k_n}) = 0$,
	\item $g$ is smoothed near the endpoints to ensure $C^\infty$ regularity.
\end{itemize}

For sufficiently small $\delta>0$, this construction makes each $l_i$ an attracting invariant curve. 
Consequently, the time averages of almost every orbit converge to a measure supported entirely on $\bigcup_i l_i$, which coincides with the previously defined measure $\mathfrak M_{\delta_n}$.
Thus $\mathfrak M_{\delta_n}$ is a physical measure and satisfies the lower bound \eqref{E9}. 
This completes the proof of Lemma \ref{Thm2}.
\end{proof}

\medskip

Based on Lemma~\ref{Thm2}, we obtain a lower bound on the \( W_1 \)-distance of the Mather measures:
\begin{thm}\label{Thm5}
	Let $n=2$ and $V_0=\omega=(1,\nu)$, where $\nu$ is Diophantine with exponent $1\leq \tau(\nu)<\infty$.
	Then, for any $0<r<\tau(\nu)$, there exist a sequence $\{\delta_j>0\}_{j\in\mathbb{N}}$ tending to $0$ as $j\to\infty$ and a corresponding sequence of
	vector field  $V_{\delta_j}$ satifying 
	\[
	\sup_{x\in\mathbb{T}^2}|V_0-V_{\delta_j}(x)|=\delta_j,
	\]
such that 
	\[
	W_1\bigl(\widetilde{\mathfrak{M}}_0,\,\widetilde{\mathfrak{M}}_{\delta_j}\bigr)
	\ge c\,\delta_j^{1/(r+1)},
	\]
	where $c>0$ is a constant independent of $j\in\N$.
\end{thm}

\begin{proof}
We lift the projected Mather measure \(\mathfrak{M}_{\delta_n}\) from Lemma~\ref{Thm2} to the tangent bundle:
\[
\widetilde{\mathfrak{M}}_{\delta_n} := \frac{1}{k_n} \sum_{i=1}^{k_n} \left( \mathfrak{m}|_{l_i} \otimes \delta_{V_{\delta_n}} \right),
\]
where $\mathfrak{m}|_{l_i}$ is the normalized Lebesgue measure on the line $l_i$, and $\delta_{V_{\delta_n}}$ is the Dirac measure at the constant velocity $V_{\delta_n} = (1, p_n/k_n)$. The unperturbed Mather measure is $\widetilde{\mathfrak{M}}_0 := \mathfrak{m} \otimes \delta_{V_0}$, with $V_0 = (1, \nu)$. By the graph property, $\widetilde{\mathfrak{M}}_{\delta_n}$ is indeed the Mather measure associated with $L_{\delta_n}$, since it is supported on the graph of the flow induced by \( V_{\delta_n} \).

To estimate $W_1(\widetilde{\mathfrak{M}}_0, \widetilde{\mathfrak{M}}_{\delta_n})$, consider an optimal transport map $T(x, v) = (T_1(x), T_2(x,v))$ pushing $\widetilde{\mathfrak{M}}_0$ to $\widetilde{\mathfrak{M}}_{\delta_n}$, where $T_1(x)$ maps $x$ to its nearest point on $\bigcup_i l_i$, and $T_2(x,v) \equiv V_{\delta_n}$, which replaces the constant velocity \( V_0 \) with the perturbed vector \( V_{\delta_n} \). Then
\[
W_1(\widetilde{\mathfrak{M}}_0, \widetilde{\mathfrak{M}}_{\delta_n}) = 
 \int_{T\mathbb{T}^2} | T(x, v) - (x, v) | \, d\widetilde{\mathfrak{M}}_0 =
\int_{\mathbb{T}^2} \sqrt{ |T_1(x) - x|^2 + |V_{\delta_n} - V_0|^2 } \, dx.
\]
Using $\sqrt{a^2+b^2} \ge (a+b)/\sqrt{2}$ for $a,b\ge0$, we get
\[
W_1(\widetilde{\mathfrak{M}}_0, \widetilde{\mathfrak{M}}_{\delta_n}) \ge \frac{1}{\sqrt{2}} \left( \int_{\mathbb{T}^2} |T_1(x)-x|\,dx + |V_{\delta_n}-V_0| \right) = \frac{1}{\sqrt{2}} \left( W_1(\mathfrak{M}_0, \mathfrak{M}_{\delta_n}) + \delta_n \right).
\]
Applying Lemma~\ref{Thm2} yields the desired lower bound:
\[
W_1(\widetilde{\mathfrak{M}}_0, \widetilde{\mathfrak{M}}_{\delta_n}) 
\geq \frac{1}{\sqrt{2}} \left( \frac{1}{4} \delta_n^{\frac{1}{r+1}} + \delta_n \right)
\sim \mathcal{O}\left( \delta_n^{\frac{1}{r+1}} \right)
\]
as $n\rightarrow +\infty$.
\end{proof}

\begin{rmk}\label{rmk:2}
In the theorem above, the lower bound exponent is $1/(r+1)$, which cannot be less than $1/2$ (since $\tau(\nu) \ge 1$). Meanwhile, Theorem~\ref{Thm1} provides an upper bound exponent $l \le 1/3$ for the case $\tau=1, n=2$. Thus there is a gap between the exponents obtained from the upper and lower bounds. 
\end{rmk}

\subsection{Lower bound in higher dimensions ($n>2$)}\label{Sec:3.3}
The previous argument suggests a possible lower bound for $n>2$ under additional arithmetic assumptions. While a general lower bound is not available, we present a result under a restrictive arithmetic condition.

Let the non-resonant vector be 
\[
\omega = (\omega_1, \omega_2, \dots, \omega_{n-1}, 1) \in \mathbb{R}^n,
\]
and assume that each component $\omega_k$ has Diophantine exponent $\tau(\omega_k) = 1$ for $k=1,\dots,n-1$.

For each component $\omega_k$, choose a sequence of rational approximants
\[
\frac{p_{k,m}}{q_{k,m}}, \qquad p_{k,m}, q_{k,m} \in \mathbb{Z}, \ q_{k,m} > 0,
\]
and define the perturbed velocity vector
\[
V_{\delta_m} := \left( \frac{p_{1,m}}{q_{1,m}},\, \frac{p_{2,m}}{q_{2,m}},\, \dots,\, \frac{p_{n-1,m}}{q_{n-1,m}},\, 1 \right).
\]
Let $T_m = \operatorname{lcm}(q_{1,m}, \dots, q_{n-1,m})$, and set $Q_m = \max_{1\le k\le n-1} q_{k,m}$.  Since
\[
0 < {\max_{1\leq k\leq n-1} q_{k,m}} \le T_m \le \prod_{k=1}^{n-1} q_{k,m} \le \bigl(\max_{1 \le k \le n-1} q_{k,m}\bigr)^{n-1},
\]
then $Q_m \le T_m \le Q_m^{n-1}$. Define
\[
\delta_m := \max_{1\le k\le n-1} \left| \omega_k - \frac{p_{k,m}}{q_{k,m}} \right|.
\]

Proceeding as in the proof of Theorem~\ref{Thm5}, we obtain
a lower bound for the Wasserstein distance between the Mather measures:
\begin{equation}\label{eq:lowerbound_W1_prelim}
	W_1\bigl(\widetilde{\mathfrak{M}}_0, \widetilde{\mathfrak{M}}_{\delta_m}\bigr)
	\;\ge\;
	C \Bigl(
	W_1\bigl(\mathfrak{M}_0, \mathfrak{M}_{\delta_m}\bigr)
	+ | V_{\delta_m} - V_0 | \Bigr),
\end{equation}
where \(C>0\) denotes a constant independent of \(m\). With the above choice of rational approximations, one can construct periodic orbits whose associated Poincar\'e sections are divided into \(T_m\) equal parts within a single period. Therefore,
\begin{equation}\label{eq:lowerbound_base}
	W_1\bigl(\mathfrak{M}_0, \mathfrak{M}_{\delta_m}\bigr)
	\;\ge\; \frac{C}{T_m^{1/(n-1)}}
	\;\ge\; \frac{C}{Q_m}.
\end{equation}
On the other hand, the velocity difference is bounded, i.e.,
\begin{equation}\label{eq:lowerbound_velocity}
	\delta_m \;\le\; | V_{\delta_m} - V_0 | \;\le\; \sqrt{n-1}\,\delta_m.
\end{equation}
Combining \eqref{eq:lowerbound_W1_prelim}, \eqref{eq:lowerbound_base}, and \eqref{eq:lowerbound_velocity}, we obtain
\begin{equation}\label{eq:final_lowerbound_W1}
	W_1\bigl(\widetilde{\mathfrak{M}}_0, \widetilde{\mathfrak{M}}_{\delta_m}\bigr)
	\;\ge\; C \bigl( Q_m^{-1} + \delta_m \bigr),
\end{equation}
where the right hand side tends to zero as $m\rightarrow+\infty$.

Because each $\omega_k$ has Diophantine exponent $1$, we can choose the approximants so that $\delta_m \sim Q_m^{-2}$. Consequently,
\[
W_1\bigl(\widetilde{\mathfrak{M}}_0, \widetilde{\mathfrak{M}}_{\delta_m}\bigr) \ge C \delta_m^{1/2}.
\]

This gives a Hölder exponent of $1/2$ under these specific arithmetic conditions, which is larger than the exponents obtained in the two-dimensional case. This suggests that the regularity of Mather measures may degrade as the dimension increases, even under favorable arithmetic conditions.

\bigskip

\section{Statistical regularity of Mather measures for Tonelli Hamiltonians under perturbations}
\label{Sec: generalcase}

Motivated by the results obtained in the previous section, we now turn to a more general class of Tonelli Hamiltonian systems. We consider the parametrized Hamilton--Jacobi equation
\begin{equation}\label{H-Jepsilon}
	H_\varepsilon(x,du_\varepsilon(x)) = \alpha(\varepsilon), \qquad x \in \mathbb{T}^n,
\end{equation}
where \(\alpha(\varepsilon)\) denotes the ergodic constant associated with \(H_\varepsilon\) and \(\varepsilon\in\R^{m}\) (\(m\in\N\)).

We add the following assumptions:

\begin{enumerate}[label=(A\arabic*),ref=A\arabic*]
	\setcounter{enumi}{3}
	\item \label{ass:A4}
	 \(H_\varepsilon\) remains Tonelli for \(\varepsilon\) with sufficiently small norm. Moreover, \(H_\varepsilon \to H_0\) in the \(C^2\) topology as \(|\varepsilon| \to 0\).
	 
	\item \label{ass:A5} Equation \eqref{H-Jepsilon} associated with \(\varepsilon=0\) admits a classical solution \(u_0 \in C^{1,1}\), such that 
	\[
	\widetilde{\mathcal{T}}_\omega := \{ (x, du_0(x)) : x \in \mathbb{T}^n \}
	\]
	is a KAM torus with frequency \(\omega\in \mathcal{D}_{\sigma,\tau}\).
\end{enumerate}

\begin{defn}[KAM torus]
	A torus \(\widetilde{\mathcal{T}}_\omega\) is called a \emph{KAM torus} for $H_0$ with frequency \(\omega \in \mathbb{R}^n\) if it is invariant under the Hamiltonian flow \(\varphi_{H_0}^t\), and there exists a diffeomorphism \(\psi : \mathbb{T}^n \to \mathbb{T}^n\) such that for any \(x\in\T^n\),
	\[
	\psi\!\big(\pi \circ \varphi_{H_0}^t(x,du_0(x))\big) = \psi(x) + \omega t \pmod{\Z^n}.
	\]
	Equivalently, the following diagram commutes:
\begin{equation}\label{eq:conjugacy-diagram}
	\begin{tikzcd}
		\widetilde{\mathcal{T}}_\omega \arrow[r, "\varphi_{H_0}^t"] \arrow[d, "\psi\circ \pi"']
		& \widetilde{\mathcal{T}}_\omega \arrow[d, "\psi\circ \pi"] \\
		\mathbb{T}^n \arrow[r, "\rho_\omega^t"]
		& \mathbb{T}^n
	\end{tikzcd}
\end{equation}
\end{defn}

We recall several standard results from weak KAM theory that will be used repeatedly. We refer to \cite{Fathi05cv,Fathiweakkam,BernardC11,Davini06,DaviniFathi2016} for the proofs.
\begin{lem}\label{lemma:basic weakKAM}
	For the Hamilton--Jacobi equation \eqref{H-J} with a Tonelli Hamiltonian \(H\), we have the following results.
	
	\begin{enumerate}
		\item Any critical subsolution is differentiable on \(\mathcal{A}\), and it is in fact a solution of the Hamilton--Jacobi equation \eqref{H-J} there.
		\item Any two critical subsolutions differ by an additive constant on each connected component of the Aubry set; the constant may differ between different components.
		\item There exists a \(C^{1,1}\) critical subsolution of \eqref{H-J} which is strict outside the Aubry set \(\mathcal{A}\).
		\item The map \(\pi^{-1} : \mathcal{A} \to \widetilde{\mathcal{A}}\) is Lipschitz.
	\end{enumerate}
\end{lem}

Among the various perturbations of Tonelli Hamiltonians, two types are of most interest for our analysis (as already mentioned in the Introduction):
\begin{enumerate}[label=(\Roman*)]
	\item {\it Cohomological parametrization}
	\[
	H_c(x,p) = H(x, p + c), \qquad c \in H^1(\mathbb{T}^n,\mathbb{R})
	\]
	introduced by Mather in 1991~\cite{Mather91}, where \(|c| \ll 1\);
	
	\item {\it Ma\~n\'e's potential perturbation}
	\[
	H_\varepsilon(x,p) = H(x,p) + \varepsilon f(x), \qquad 0\leq \varepsilon \ll 1,
	\]
	proposed by Ma\~{n}\'{e} in 1992~\cite{Mane92} with \(f \in C^{2}(\mathbb{T}^n,\mathbb{R})\).
\end{enumerate}

We will respectively analyze these two cases and investigate the stability of the Mather measures. Throughout this section, we denote by \(\widetilde{\mathfrak{M}}(0)\), \(\widetilde{\mathfrak{M}}(\varepsilon)\) and \(\widetilde{\mathfrak{M}}(c)\) the sets of Mather measures associated with \(H\), \(H_\varepsilon\) and \(H_c\), respectively.  \(\widetilde{\mathfrak{M}}(0)\) is a singleton due to assumption \eqref{ass:A5}.

\subsection{Case I: Cohomological perturbations}
\label{subsection4.2}
In this subsection, we study \eqref{H-Jepsilon} in the case where it takes the following form:
\begin{equation}\label{E14}
	H\big(x, du_0(x)\big) = \alpha(0),
\end{equation}
together with its perturbation
\begin{equation}\label{E15}
	H\big(x, c + du_c(x)\big) = \alpha(c),
\end{equation}
where \(c \in \mathbb{R}^n\) with \(|c|\ll 1\).

\begin{thm}\label{Thm4}
	Assume that Assumptions \eqref{ass:A4} and \eqref{ass:A5} hold.
	For the Hamilton--Jacobi equation of type (I), we have
	\begin{equation}\label{E50}
		W_1\big(\widetilde{\mathfrak{M}}(c), \widetilde{\mathfrak{M}}(0)\big)
		\le C\,|c|^{\,\ell},
	\end{equation}
	where \(\ell=\frac{1}{\,n+\tau+1\,}\) and
	\(C=C(n,\omega,H_0)>0\) is a constant.
\end{thm}

\proof
By the Kantorovich--Rubinstein duality formula (see Remark \ref{rmk:optimaltransport}), we have  
\[
W_1\!\left( \widetilde{\mathfrak{M}}(c), \widetilde{\mathfrak{M}}(0) \right) 
= \sup_{\phi \in \mathrm{Lip}_1(T\mathbb{T}^n )} \left\{ 
\int_{T\mathbb{T}^n} \phi \, d\widetilde{\mathfrak{M}}(c) 
- \int_{T\mathbb{T}^n} \phi \, d\widetilde{\mathfrak{M}}(0) 
\right\}.
\]
Let \(u_c\) denote the viscosity solution of \eqref{E15}, and let 
\(\mathcal{A}(c)\) be the associated projected Aubry set. By 
Lemma~\ref{lemma:basic weakKAM}, the Mather set is contained in the Aubry set, 
and both 
\(d u_c\) and \(d u_0\) are well defined on \(\mathcal{A}(c)\) and 
\(\mathcal{A}(0)\).  
Then by the graph property of Mather measures, the above formula reduces to  
\[
\begin{aligned}
	W_1\!\big( \widetilde{\mathfrak{M}}(c), \widetilde{\mathfrak{M}}(0) \big) 
	&\leq \sup_{\phi \in \mathrm{Lip}_1(\mathbb{T}^n)} \Bigg\{ 
	\int_{\mathcal{A}(c)} 
	\phi\!\big(x, \partial_p H\big(x, d u_c(x)+ c\big)\big) 
	\, d\mathfrak{M}(c) \\
	&\quad - \int_{\mathcal{A}(0)} 
	\phi\!\big(x, \partial_p H\big(x, d u_0(x)\big)\big) 
	\, d\mathfrak{M}(0) 
	\Bigg\},
\end{aligned}
\]
where \( \mathfrak{M}(c) \) and \( \mathfrak{M}(0) \) denote the projected Mather measures associated with 
\( \widetilde{\mathfrak{M}}(c) \) and \( \widetilde{\mathfrak{M}}(0) \), respectively.
For notational convenience, we write
\[
V_c(x) := \partial_p H\!\big(x, d u_c(x)+ c\big), 
\quad x \in \mathcal{A}(c),
\]
and
\[
V_0(x) := \partial_p H\!\big(x, d u_0(x)\big), 
\quad x \in \mathbb{T}^n.
\]
By Lemma~\ref{lemma:basic weakKAM}, it suffices to consider these vector fields on the Aubry set:
\begin{align}\label{eqkeyinequality}
	W_1\!\left(\widetilde{\mathfrak{M}}(c), \widetilde{\mathfrak{M}}(0)\right)
	&\leq \sup_{\phi \in \mathrm{Lip}_1(\mathbb{T}^n )}
	\left\{
	\bigg[\int_{\mathcal{A}(c)} \phi\big(x, V_c(x)\big)\, d\mathfrak{M}(c)
	-\int_{\mathcal{A}(c)} \phi\big(x, V_0(x)\big)\, d\mathfrak{M}(c)\bigg]\right. \nonumber\\
	&\hspace{1.2cm}\left.
	+\bigg[\int_{\mathbb{T}^n} \phi\big(x, V_0(x)\big)\, d\mathfrak{M}(c)
	-\int_{\mathbb{T}^n} \phi\big(x, V_0(x)\big)\, d\mathfrak{M}(0)\bigg]
	\right\}.
\end{align}

We estimate the two bracketed terms separately. For the first term in inequality \eqref{eqkeyinequality}, since 
\(\phi \in \mathrm{Lip}_1(\mathbb{T}^n)\) and \(H\) is smooth, we have
\begin{align*}
\sup_{x\in \mathcal{A}(c)} \big| \phi(x, V_c(x)) - \phi(x, V_0(x)) \big| 
&\le \sup_{x\in \mathcal{A}(c)} \big| V_c(x) - V_0(x) \big| \\
&\le C \sup_{x\in \mathcal{A}(c)} \big| d u_c(x) + c - d u_0(x)  \big|.
\end{align*}
Here and in what follows, \(C\) denotes a positive constant that may change from line to line, but it is independent of \(c\).

By Theorem 2.1 of \cite{YanETDS2014}, for any sufficiently small \(\epsilon_0 > 0\), 
the viscosity solution \(u_c\) of \eqref{E15} satisfies the following estimate 
at all points where it is differentiable:
\[
\big| c + d u_c - d u_0 \big| 
\le C\, \|c\|^{1 - \epsilon_0}.
\]  
We immediately obtain
\begin{align}\label{estimateforVepsilon}
\sup_{x\in \mathcal{A}(c)}\|V_c(x)-V_0(x)\|
\le C\, |c|^{1 - \epsilon_0}.
\end{align}
Hence, we have
\[
\int_{\mathbb{T}^n} \phi(x, V_c(x)) \, d\mathfrak{M}(c) - \int_{\mathbb{T}^n} \phi(x, V_0(x)) \, d\mathfrak{M}(c) \leq \mathcal{O}(|c|^{1-\epsilon_0}).
\]
For the second term in inequality \eqref{eqkeyinequality}, note that it can be written as
\[
\sup_{\phi\in \mathrm{Lip}_1(\mathbb{T}^n)}\int_{\mathbb{T}^n} \phi(x, V_0(x)) \, d\mathfrak{M}(c)
- \int_{\mathbb{T}^n} \phi(x, V_0(x)) \, d\mathfrak{M}(0)
\;=\;
W_1\!\left(\mathfrak{M}(c),\,\mathfrak{M}(0)\right).
\]
By (3) of Lemma~\ref{lemma:basic weakKAM},
there exists a \(C^{1,1}\) critical subsolution \(\omega_c\) of \eqref{H-Jepsilon} which is strict outside the projected Aubry set. Although such a subsolution could be chosen rather random, we can impose $\om_c(x_0)=0$ at a fixed point $x_0\in\T^n$ and due to the superlinearity of $H(x,p)$, all possible $\om_c$ are uniformly Lipschitz and then uniformly bounded on $\T^n$. Moreover, due to (1) of Lemma~\ref{lemma:basic weakKAM} and Lemma \ref{lem:mathersetisinAubry}, the multiplicity of $\om_c$ makes no ambiguity on $\cM$. This allows us to define the Lipschitz vector field 
\be\label{eq:Ve-def}
\wh V_c : \T^n \to \R^n, \quad \wh V_c(x) = \partial_p H_{c}\!\big(x, d\omega_c(x)\big),
\ee
which yields a Ma\~{n}\'e-type modified Lagrangian
\begin{equation}\label{eq:Lsub}
	L_{\mathrm{sub}}^c(x,v) := \frac{1}{2}\,\big|v - \wh V_c(x)\big|^2.
\end{equation}
We denote by \(\wt{\mathfrak{M}}_{\mathrm{sub}}(c)\) the set of Mather measures associated with \(L_{\mathrm{sub}}^c\); then by part (1) of Lemma \ref{lemma:basic weakKAM}, we actually have \(\wt{\mathfrak M}(c)\subset\wt{\mathfrak  M}_{\rm sub}(c)\). By the triangle inequality for \(W_1\), inserting the projected Mather measure
\(\mathfrak{M}_{\mathrm{sub}}(c)\) associated with the Ma\~n\'e-type Lagrangian \(L_{\mathrm{sub}}\), we have
\[
W_1\!\left(\mathfrak{M}(c),\mathfrak{M}(0)\right)
\le W_1\!\left(\mathfrak{M}(c),\mathfrak{M}_{\mathrm{sub}}(c)\right)
+ W_1\!\left(\mathfrak{M}_{\mathrm{sub}}(c),\mathfrak{M}(0)\right).
\]
By Lemma~\ref{lemma:basic weakKAM}, we have \(V(c)=W(c)\) on \(\mathcal{A}(c)\). Hence the first term on the right-hand side vanishes:
\[
W_1\!\left(\mathfrak{M}(c),\,\mathfrak{M}_{\mathrm{sub}}(c)\right)=0.
\]
For the second term, using \eqref{estimateforVepsilon} and Theorem \ref{Thm1}, we have
\[
W_1\!\left(\mathfrak{M}_{\mathrm{sub}}(c),\mathfrak{M}(0)\right)
\leq \mathcal{O}\!\left(|c|^{\ell_2}\right)
\]
for some \(\ell_2\in(0,1)\). This proves the desired bound for the second term in inequality \eqref{eqkeyinequality}. By choosing \(\epsilon_0\) sufficiently small and combining estimate \eqref{estimateforVepsilon} with the bound obtained in Remark~\ref{rmkA2}, we deduce that \eqref{E50} holds with 
\(\ell = \frac{1}{\,n+\tau+1\,}.\) This completes the proof.\qed

\begin{rmk}\label{rmk:DiophantineImproved}
	In \cite{Bolotin03openproblem}, Walter Craig proposed the following open problem:
	\begin{quote}
 For a Tonelli Lagrangian \(L(x,v)\) and the associated \(\alpha\)-function \(\alpha: c\in H^1(\T^n,\R)\to \R\), how can we relate the regularity of \(\wt{\mathfrak M}(c)\) (with respect to \(c\)) to the arithmetic properties of \(\omega\in D^*\alpha(c) \subset H_1(\T^n,\R)\)?
	 \end{quote}
	 As we can see, Theorem \ref{Thm4} partially answers this problem by establishing the Hölder continuity of \(\wt{\mathfrak M}(c)\) with respect to \(c\in B(0,\delta)\) when \(\nabla\alpha(0)=\omega\in\mathcal D_{\sigma,\tau}\).
\end{rmk}

\subsection{Case II: Ma\~n\'e's perturbations}
In this subsection,
we consider perturbed Hamiltonians given by  
\[
H_\varepsilon(x, p) = H(p) + \varepsilon f(x), \qquad x \in \mathbb{T}^n,
\]
with \(0 \leq \varepsilon \ll 1\), where \(H : \mathbb{R}^n \to \mathbb{R}\) is a Tonelli Hamiltonian that depends only on the momentum variable \(p\), and \(f : \mathbb{T}^n \to \mathbb{R}\) is a \(C^2\)-smooth perturbation. 

\begin{thm}\label{Thm3}
	Suppose that Assumptions \eqref{ass:A4} and \eqref{ass:A5} hold, and let \(f \in C^{2}(\mathbb{T}^n,\mathbb{R})\).
	 Then there exists a modulus of continuity \(\Theta : \mathbb{R}^+ \to \mathbb{R}^+\) with \(\lim_{r\to 0^+}\Theta(r)=0\) such that
	\[
	W_1\!\bigl(\widetilde{\mathfrak{M}}(\varepsilon), \widetilde{\mathfrak{M}}(0)\bigr) \;\leq\; \Theta(\varepsilon).
	\]
\end{thm}

\proof
According to the Kantorovich--Rubinstein duality formula (see Remark \ref{rmk:optimaltransport}), we have
\[
	W_1\!\bigl(\widetilde{\mathfrak{M}}(\varepsilon),\widetilde{\mathfrak{M}}(0)\bigr)
	= \sup_{\Phi\in\mathrm{Lip}_1(T\mathbb{T}^n)}
	\left\{ \int \Phi\,d\widetilde{\mathfrak{M}}(\varepsilon)-\int \Phi\,d\widetilde{\mathfrak{M}}(0) \right\}.
\]
By the graph property of Mather measures (Lemma~\ref{lemma:basic weakKAM}), each Mather measure is supported on a Lipschitz graph. Hence there exist vector fields
	\[
	V_\varepsilon(x):=\partial_p H_\varepsilon\big(x,du_\varepsilon(x)\big),\qquad
	V_0(x):=\partial_p H_0\big(x,du_0(x)\big),
	\]
	defined on the respective projected Aubry sets. Therefore the dual formula reduces to
	\[
	W_1\!\bigl(\widetilde{\mathfrak{M}}(\varepsilon),\widetilde{\mathfrak{M}}(0)\bigr)
	= \sup_{\phi\in\mathrm{Lip}_1(\mathbb{T}^n)}
	\Bigg\{ \int_{\mathcal{A}_\varepsilon} \phi(x,V_\varepsilon(x))\,d\mathfrak{M}(\varepsilon)
	- \int_{\mathcal{A}_0} \phi(x,V_0(x))\,d\mathfrak{M}(0) \Bigg\}
	\]
	and then, adding and subtracting an intermediate term,
	\begin{align*}
		W_1\!\bigl(\widetilde{\mathfrak{M}}(\varepsilon),\widetilde{\mathfrak{M}}(0)\bigr)
		\le{}& \sup_{\phi\in\mathrm{Lip}_1(\mathbb{T}^n)} \Bigg\{
		\Big[\int_{\mathcal{A}_\varepsilon}\big(\phi(x,V_\varepsilon)-\phi(x,V_0)\big)\,d\mathfrak{M}(\varepsilon)\Big] \\
		&\qquad +\Big[\int_{\mathbb{T}^n}\phi(x,V_0)\,d\mathfrak{M}(\varepsilon)
		- \int_{\mathbb{T}^n}\phi(x,V_0)\,d\mathfrak{M}(0)\Big]\Bigg\}.
	\end{align*}
	The first bracket is clearly bounded by
	\[
	\Big|\int_{\mathcal{A}_\varepsilon}\big(\phi(x,V_\varepsilon)-\phi(x,V_0)\big)\,d\mathfrak{M}(\varepsilon)\Big|
	\le \sup_{x\in\mathcal{A}_\varepsilon}\big|V_\varepsilon(x)-V_0(x)\big|.
	\]
	The second bracket equals the \(W_1\)-distance between the projected Mather measures \(\mathfrak{M}(\varepsilon)\) and \(\mathfrak{M}(0)\); hence
	\[
	\sup_{\phi\in\mathrm{Lip}_1(\mathbb{T}^n)}
	\Big\{\int_{\mathbb{T}^n}\phi(x,V_0)\,d\mathfrak{M}(\varepsilon)
	- \int_{\mathbb{T}^n}\phi(x,V_0)\,d\mathfrak{M}(0)\Big\}
	= W_1\!\bigl(\mathfrak{M}(\varepsilon),\mathfrak{M}(0)\bigr).
	\]
By Lemma~\ref{lemma:basic weakKAM},
there exists a \(C^{1,1}\) critical subsolution \(\omega_\varepsilon\) of \eqref{H-Jepsilon} which is strict outside the projected Aubry set. 
This allows us to define the Lipschitz vector field
\be\label{eq:Ve-def-1}
W_\varepsilon : \T^n \to \R^n, \quad W_\varepsilon(x) = \partial_p H_{\varepsilon}\!\big(x, d\omega_\varepsilon(x)\big),
\ee
which yields a Ma\~{n}\'e-type modified Lagrangian
\begin{equation}\label{eq:Lsub-1}
	L_{\mathrm{sub}}^\varepsilon(x,v) := \frac{1}{2}\,\big|v - W_\varepsilon(x)\big|^2.
\end{equation}
We denote by \(\wt{\mathfrak{M}}_{\mathrm{sub}}(\varepsilon)\) the set of Mather measures associated with \(L_{\mathrm{sub}}^\varepsilon\); then by part (1) of Lemma \ref{lemma:basic weakKAM}, we actually have \(\wt{\mathfrak M}(\varepsilon)\subset\wt{\mathfrak  M}_{\rm sub}(\varepsilon)\). 
By introducing the Ma\~n\'e-type Lagrangian \(L_{\mathrm{sub}}\) in \eqref{eq:Lsub-1} and the associated projected Mather measure \(\mathfrak{M}_{\mathrm{sub}}(\varepsilon)\), and combining Lemma~\ref{lemma:basic weakKAM} and Theorem~\ref{Thm1}, 
there exist a constant \(C>0\) independent of \(\varepsilon\) and a constant \(\ell \in (0,1)\) such that
\[
W_1\bigl(\mathfrak{M}(\varepsilon),\mathfrak{M}(0)\bigr) 
\le C \sup_{x\in\mathcal{A}_\varepsilon} \big|V_\varepsilon(x)-V_0(x)\big|^\ell.
\]
Combining the two estimates, we have 
	\[
	W_1\!\bigl(\widetilde{\mathfrak{M}}(\varepsilon),\widetilde{\mathfrak{M}}(0)\bigr)
	\le C \,\sup_{x\in\mathcal{A}_\varepsilon}
	\bigl|\partial_p H_\varepsilon(x,du_\varepsilon(x)) - \partial_p H_0(x,du_0(x))\bigr|^\ell.
	\]
	In the following, we estimate
	\[
	\sup_{x\in\mathcal{A}_\varepsilon}
	\bigl|\partial_p H_\varepsilon(x,du_\varepsilon(x)) - \partial_p H_0(x,du_0(x))\bigr|.
	\]
Since the Hamiltonian \(H_0\) is smooth in the momentum variable, there exists a constant \(C>0\), independent of \(\varepsilon\), such that
	\[
	\sup_{x\in\mathcal{A}_\varepsilon}
	\bigl|\partial_p H_\varepsilon(x,du_\varepsilon(x)) - \partial_p H_0(x,du_0(x))\bigr|
	\le C\sup_{x\in\mathcal{A}_\varepsilon}|du_\varepsilon(x)-du_0(x)|.
	\]
Notice that for every \(x\) where \(u_\varepsilon\) is differentiable we have \(du_\varepsilon(x)= D^*u_\varepsilon(x)\); hence
\[
\sup_{x\in\mathcal{A}_\varepsilon}|du_\varepsilon(x)-du_0(x)|
\le \sup_{x\in\mathbb{T}^n} d_H\bigl(D^*u_\varepsilon(x),du_0(x)\bigr),
\]
where \(d_H(\cdot, \cdot)\) denotes the Hausdorff distance\footnote{Let \(A,B \subset T^*\mathbb{T}^n\) be two nonempty compact sets. The \emph{Hausdorff distance} between \(A\) and \(B\) is defined by
	\[
	d_H(A,B) := \max \big\{ \rho(A,B), \rho(B,A) \big\}, \qquad
	\rho(A,B) := \sup_{a \in A} \inf_{b \in B} |a-b|_x.
	\] } in \(T^* \mathbb{T}^n\).
Combining the above inequalities yields
	\[
	\sup_{x\in\mathcal{A}_\varepsilon}
	\bigl|\partial_p H_\varepsilon(x,du_\varepsilon(x)) - \partial_p H_0(x,du_0(x))\bigr|
	\le C\sup_{x\in\mathbb{T}^n} d_H\bigl(D^*u_\varepsilon(x),du_0(x)\bigr),
	\]
	where \(C>0\) is independent of \(\varepsilon\) (and may change from line to line). 
	In the following, we claim that
	\[
	\lim_{\epsilon \to 0^+} \sup_{x \in \mathbb{T}^n} d_H(D^* u_\epsilon(x), du_0(x)) = 0.
	\]
	To see this, for any \(p \in D^* u_\epsilon(x)\), by weak KAM theory there exists a \(u_\epsilon\)-calibrated curve \(\gamma_\epsilon : (-\infty, 0] \to \mathbb{T}^n\) with \(\gamma_\epsilon(0) = x\) satisfying the calibration property
	\be\label{eq:calibrated}
	u_\epsilon(x) - u_\epsilon(\gamma_\epsilon(-t)) = \int_{-t}^0 L_\epsilon(\gamma_\epsilon(s), \dot{\gamma}_\epsilon(s)) \, ds \quad \text{for all } t \ge 0.
	\ee
	By Theorem 4.3.8 in \cite{Fathiweakkam}, each \(u_\varepsilon\)-calibrated curve satisfies
	\[
	\dot\gamma_\varepsilon(t) = \partial_p H_\varepsilon\bigl(\gamma_\varepsilon(t),du_\varepsilon(\gamma_\varepsilon(t))\bigr),
	\qquad -\infty < t < 0.
	\]
By the calibration identity \eqref{eq:calibrated} and the superlinearity of \(L_\varepsilon\), the family \(\{\gamma_\varepsilon\}_{0<\varepsilon \ll 1}\) is uniformly bounded and equi-Lipschitz. Hence, any limit curve of \(\gamma_\varepsilon\) as \(\varepsilon \to 0\) is \(C^1\) and \(u_0\)-calibrated. Since \(u_0\) is smooth, such a limit curve is unique, and its initial velocity corresponds exactly to the covector \(du_0(x)\). Therefore,
\[
\lim_{\varepsilon \to 0^+} d_H\bigl(D^*u_\varepsilon(x), \{du_0(x)\}\bigr) = 0,
\]
that is, \(D^*u_\varepsilon(x)\) converges in Hausdorff distance to \(\{du_0(x)\}\).

	We define
	\[
	\theta(\epsilon) := \sup_{x \in \mathbb{T}^n} d_H(D^* u_\epsilon(x), du_0(x)), \qquad \epsilon \ge 0,
	\]
	and then set
	\[
	\eta(r) := \sup_{0 \le s \le r} \theta(s), \qquad r \ge 0.
	\]
	By construction, \(\eta\) is nondecreasing and satisfies \(\eta(r) \ge \theta(r)\) for all \(r \ge 0\).
	Since \(\eta:[0,\varepsilon_0]\to \R^+\) is monotone increasing and bounded, we define
	\[
	\Theta(r):=\int_0^1 \eta\bigl(r(1+t)\bigr)\varphi(t)\,\mathrm{d}t, \qquad r\ge0,
	\]
	where \(\varphi \in C^\infty([0,1])\) is a nonnegative mollifier with \(\int_0^1 \varphi = 1\).	
	It is clear that \(\Theta\) dominates \(\eta\). Indeed, since \(r(1+t)\ge r\) for all \(t\in[0,1]\) and \(\eta\) is monotone, we have \(\eta(r(1+t))\ge \eta(r)\); hence
	\[
	\Theta(r)\;=\;\int_0^1 \eta\bigl(r(1+t)\bigr)\varphi(t)\,\mathrm{d}t
	\;\ge\;\int_0^1 \eta(r)\varphi(t)\,\mathrm{d}t
	\;=\;\eta(r).
	\]	
	Monotonicity follows immediately. Whenever \(0\le r_1<r_2\), one has \(r_1(1+t)<r_2(1+t)\) for every \(t\), and therefore \(\eta(r_1(1+t))\le \eta(r_2(1+t))\). Integrating against \(\varphi\) gives \(\Theta(r_1)\le \Theta(r_2)\).  
	To establish continuity, fix \(r_0>0\). Since \(\eta\) has at most countably many discontinuities, the set
	\[
	N=\{t\in[0,1]:\; r_0(1+t)\text{ is a discontinuity of }\eta\}
	\]
	has Lebesgue measure zero. For \(r_n\to r_0\), the functions \(h_n(t):=\eta(r_n(1+t))\) converge pointwise to \(h(t):=\eta(r_0(1+t))\) on \([0,1]\setminus N\). By Egorov's theorem, for any \(\delta_1>0\) there exists a measurable set \(E\) with \(m(E)<\delta_1\) such that the convergence is uniform on \(F:=[0,1]\setminus E\). On the other hand, since \(\varphi\in L^1([0,1])\), the absolute continuity of the Lebesgue integral ensures that there exists \(\delta_1>0\) such that if \(m(A)<\delta_1\) then
	\[
	\int_A \varphi(t)\,dt < \frac{\varepsilon}{4M}, \qquad M=\sup_{[0,\varepsilon_0]}\eta.
	\]
	Therefore, for \(|r-r_0|\) sufficiently small, one obtains
	\[
	|\Theta(r)-\Theta(r_0)| 
	= \left|\int_0^1 \bigl(\eta(r(1+t))-\eta(r_0(1+t))\bigr)\varphi(t)\,dt\right|
	\le I_F+I_E,
	\]
	where
	\[
	I_F := \int_F |\eta(r(1+t))-\eta(r_0(1+t))|\varphi(t)\,dt, 
	\quad
	I_E := \int_E |\eta(r(1+t))-\eta(r_0(1+t))|\varphi(t)\,dt.
	\]
	On \(F\), uniform convergence gives
	\[
	I_F \le \sup_{t\in F} |\eta(r(1+t))-\eta(r_0(1+t))| \cdot \int_F \varphi(t)\,dt 
	\le \frac{\varepsilon}{2}\cdot 1 = \frac{\varepsilon}{2}.
	\]
	On \(E\), boundedness of \(\eta\) yields
	\[
	I_E \le 2M \int_E \varphi(t)\,dt \le 2M\cdot \frac{\varepsilon}{4M} = \frac{\varepsilon}{2}.
	\]
	Hence \(|\Theta(r)-\Theta(r_0)|<\varepsilon\), proving continuity at \(r_0>0\).
	
	Next we show that \(\Theta\) is continuous at \(0\). Note that for \(0\le r\le \tfrac{\varepsilon_0}{2}\),
	\[
	0\le \Theta(r)\le \eta(2r)\int_0^1\varphi(t)\,\mathrm{d}t=\eta(2r).
	\]
Since \(\eta(s)\to0\) as \(s\to0^+\), we conclude that \(\Theta(r)\to0\) as \(r\to0^+\). Setting \(\Theta(0)=0\) ensures continuity at \(0\). Consequently, \(\Theta:[0,\varepsilon_0]\to\mathbb{R}_+\) is continuous, nondecreasing, satisfies \(\Theta(r)\ge\eta(r)\) for all \(r\), and \(\Theta(r)\to 0\) as \(r\to 0^+\). Therefore, up to a multiplicative constant, \(\Theta^\ell\) can be taken as the required continuous modulus, which proves Theorem~\ref{Thm3}.\qed


\section{Lipschitz regularity and Linear response under the KAM theory}\label{Sec:Linearresponse}

In this section, we start with a Tonelli Hamiltonian $H\in C^\infty(\mathbb{R}^n,\mathbb{R})$  of the form $H=H(p)$, satisfying $\nabla H(0) = \omega$, where $\omega$ is a fixed Diophantine vector. Consequently, $u_0\equiv \text{const}$ is a smooth solution of 
\[
H(du_0) = \alpha(0)
\]
and 
\[
\widetilde{\mathcal{T}}_\omega := \{(x, 0) \mid x \in \mathbb{T}^n\}
\]
is a KAM torus with frequency $\omega$. We then perturb $H_0$ by 
\begin{equation}\label{eq:pert-ham}
H_\varepsilon(x,p) = H(p) + \varepsilon f(x),\quad |\varepsilon|\ll 1
\end{equation}
with $f\in C^\infty(\mathbb{T}^n,\mathbb{R})$ and consider the corresponding Hamilton--Jacobi equation
\begin{equation}\label{eq:HJCepsilon}
	H_\varepsilon\bigl(x, du_{\varepsilon} + c_\varepsilon\bigr) = \alpha(\varepsilon)
\end{equation}
for suitable $c_\varepsilon\in H^1(\mathbb{T}^n,\mathbb{R})$ and $\alpha(\varepsilon)\in\mathbb{R}$. Our goal, by applying classical KAM theory, is to select $c_\varepsilon$ so that for each sufficiently small $|\varepsilon|\ll 1$, there exists a perturbed KAM torus
\[
\widetilde{\mathcal{T}}_\omega^\varepsilon := \{(x,\,c_\varepsilon + du_\varepsilon(x)) \mid x \in \mathbb{T}^n\}
\]
which remains conjugate to the linear flow on $\mathbb{T}^n$ with frequency $\omega$. For each $\varepsilon $, the set $\widetilde{\mathfrak{M}}(\varepsilon)$ is a singleton, and so is $\mathfrak{M}(\varepsilon)$. We are interested in the linear response of the Mather measure. More precisely, we say that $\widetilde{\mathfrak{M}}(\varepsilon)$ admits a linear response if the limit 
\[
\widetilde{\mathfrak{M}}_1 := 
\lim_{\varepsilon \to 0} 
\frac{\widetilde{\mathfrak{M}}(\varepsilon)-\widetilde{\mathfrak{M}}(0)}{\varepsilon}
\]
exists in the sense of weak convergence of measures. In this case, the family of measures expands as
\[
\widetilde{\mathfrak{M}}(\varepsilon)
= \widetilde{\mathfrak{M}}(0) + \varepsilon \widetilde{\mathfrak{M}}_1 + o(\varepsilon).
\]
Equivalently, for any test function $g \in C_c^\infty(T\mathbb{T}^n,\mathbb{R})$, one has
\begin{equation}\label{def:Mtilde}
\lim_{\varepsilon \to 0}
\frac{\int g\, d\widetilde{\mathfrak{M}}(\varepsilon) - \int g\, d\widetilde{\mathfrak{M}}(0)}{\varepsilon}
= \int g \, d\widetilde{\mathfrak{M}}_1.
\end{equation}

\begin{lem}[KAM theorem]\label{lem:KAM for T^n flow}
Suppose $\omega\in\mathcal D_{\sigma,\tau}$ and $\|f\|_{C^{2\tau+3+\delta}(\mathbb{T}^n,\mathbb{R})}\leq 1$ for arbitrary $\delta>0$. There exists a positive constant $\varepsilon_0:=\varepsilon_0(\sigma,\tau,\|H\|_{C^2})<1$ such that for any $\varepsilon\in(-\eps_0,\varepsilon_0)$, the associated Hamiltonian $H_\varepsilon$ defined in \eqref{eq:pert-ham} possesses a KAM torus 
\[
\widetilde{\mathcal{T}}_\omega^\varepsilon=\{(x, p_\varepsilon(x)) \mid x\in\mathbb{T}^n\}.
\] 
Precisely, there exists a diffeomorphism 
\begin{equation}\label{eq:psi-expansion}
\psi_\varepsilon:=\operatorname{id}+\varepsilon\psi_1(\cdot,\varepsilon):\mathbb{T}^n\rightarrow\mathbb{T}^n
\end{equation}
with $\|\psi_1\|_{C^{1+\delta}(\mathbb{T}^n,\mathbb{R}^n)}\leq \kappa_1$ for some constant $\kappa_1:=\kappa_1(\varepsilon_0)>0$
such that the graph in \eqref{eq:conjugacy-diagram} is commutative in the sense that $\psi_\varepsilon \circ \pi\circ\phi_{H_\varepsilon}^t=\rho_\omega^t\circ \psi_\varepsilon \circ \pi$ for any $t\in\mathbb{R}$ and $(x, p)\in T^*\mathbb{T}^n$. Moreover, $\widetilde{\mathcal{T}}_\omega^\varepsilon$ supports a unique ergodic Mather measure $\mu_\varepsilon\in\widetilde{\mathfrak{M}}(\varepsilon)$ and
\[
p_\varepsilon(x) = c_\varepsilon + du_\varepsilon(x), \quad \forall x\in\mathbb{T}^n
\]
for a uniquely identified $c_\varepsilon$ formally expressed by 
\begin{equation}\label{def:cvarepsilon}
c_\varepsilon = \varepsilon c_1 + o(\varepsilon)
\end{equation}
and a unique classical solution $u_\varepsilon$ of \eqref{eq:HJCepsilon} (differing by constants) expressed by 
\[
u_\varepsilon = u_0 + \varepsilon u_1(x) + o(\varepsilon).
\] 
Moreover, there exists a constant $\kappa_2:=\kappa_2(\kappa_1)>0$ such that $\|u_1\|_{C^{2+\delta}(\mathbb{T}^n,\mathbb{R})}\leq \kappa_2$ and $|c_1|\leq \kappa_2$.
\end{lem}

Thislemma is a straightforward adaptation of Theorem A in \cite{Po}, and the uniqueness of the cohomology class $c_\varepsilon$ and $u_\varepsilon$ as the classical solution was proved in \cite{FathiSorrentino2009}.

\begin{prop}\label{Prop:Ceplison}
	Suppose that assumptions \eqref{ass:A4} and \eqref{ass:A5} hold. Let $f \in C^{\infty}(\mathbb{T}^n,\mathbb{R})$ be the smooth potential associated with the perturbed Hamiltonian \eqref{eq:pert-ham} and \( c_\varepsilon = \varepsilon c_1 + o(\varepsilon) \) be the cohomology class associated with the perturbed equation \eqref{eq:HJCepsilon}.
	Then we have  
	\[
	\alpha(\varepsilon) = \alpha(0) + \varepsilon\bigl([f] + \langle \omega, c_1 \rangle \bigr) + \mathcal{O}(\varepsilon^2),
	\]
	where $[f]$ denotes the average of $f(x)$ over the torus $\mathbb{T}^n$.	
\end{prop}

\begin{proof}
We formally construct a smooth function \( \vartheta_\varepsilon: \mathbb{T}^n \to \mathbb{R} \) as
\[
\vartheta_\varepsilon(x)=u_0(x)+\varepsilon \vartheta_1(x),\qquad \forall x\in \mathbb{T}^n,
\]
where \( \vartheta_1(x) \) is to be determined later.
Substituting \( \vartheta_\varepsilon(x)\) into equation \eqref{eq:HJCepsilon} and applying Taylor expansion, we obtain
\begin{align*}
	H\bigl(d\vartheta_\varepsilon + c_\varepsilon\bigr) + \varepsilon f(x)
	&= H(0) 
	+ \varepsilon \bigl\langle \partial_p H(0), d\vartheta_1 \bigr\rangle
	+ \bigl\langle \partial_p H(0), c_\varepsilon \bigr\rangle
	+ \varepsilon f(x) + \mathcal{O}(\varepsilon^2)\\
	&= \alpha(0) 
	+ \varepsilon\Bigl( \bigl\langle \omega, d\vartheta_1 \bigr\rangle 
	+ f(x) + \bigl\langle \omega, c_1 \bigr\rangle \Bigr) 
	+ \mathcal{O}(\varepsilon^2),
\end{align*}
where we have used \( c_\varepsilon = \varepsilon c_1 + o(\varepsilon) \).
If we define
\[
\vartheta_1(x) := \sum_{k \neq {0}} \frac{f_k e^{ikx}}{i \langle k, \omega\rangle},
\]
where \( f_k \) denotes the \( k \)-th Fourier coefficient of the function \( f(x) \), then the non-resonance of $\omega$ guarantees the solvability of the corresponding small divisor problem. Consequently, we have
\begin{equation}\label{ceplison}
H(d \vartheta_\varepsilon+c_\varepsilon)+\varepsilon f(x)=\alpha(0)+\varepsilon([f]+\langle c_1, \omega\rangle)+\mathcal{O}(\varepsilon^2),
\end{equation}
for any $x\in \mathbb{T}^n$.
Let \( x_{\min} \) and \( x_{\max} \) be points in \( \mathbb{T}^n \) where the function \( u_\varepsilon - \vartheta_\varepsilon \) attains its minimum and maximum, respectively. Then
\[
u_\varepsilon(x_{\min}) - \vartheta_\varepsilon(x_{\min}) \leq u_\varepsilon(x) - \vartheta_\varepsilon(x) \leq u_\varepsilon(x_{\max}) - \vartheta_\varepsilon(x_{\max}).
\]
By adding or subtracting appropriate constants to \( \vartheta_\varepsilon \), and using the definitions of subsolutions and supersolutions, we obtain the following inequalities:
\[
\left\{
\begin{aligned}
	&H(d \vartheta_\varepsilon (x_{\min}))+\varepsilon f(x_{\min}) \geq \alpha(\varepsilon),\\
	&H(d \vartheta_\varepsilon (x_{\max}))+\varepsilon f(x_{\max}) \leq \alpha(\varepsilon).
\end{aligned}
\right.
\]
Substituting equation \eqref{ceplison} into the above system of inequalities yields
\[
\alpha(0)+\varepsilon([f]+\langle c_1, \omega\rangle)+\mathcal{O}(\varepsilon^2) \geq \alpha(\varepsilon) \geq \alpha(0)+\varepsilon([f]+\langle c_1, \omega\rangle)+\mathcal{O}(\varepsilon^2).
\]
Hence, we conclude that
\[
\alpha(\varepsilon)=\alpha(0)+\varepsilon([f]+\langle c_1, \omega\rangle)+\mathcal{O}(\varepsilon^2),
\]
which completes the proof.
\end{proof}

\begin{thm}[Linear response]\label{Thm:KAM-W1}
	Assume that \eqref{ass:A4} and \eqref{ass:A5} hold. 
	Suppose that \(f \in C^\infty(\mathbb{T}^n, \mathbb{R})\) is the smooth potential associated with the perturbed Hamiltonian \eqref{eq:pert-ham}, 
	then the Mather measure $\widetilde{\mathfrak{M}}(\varepsilon)$ admits a linear response; i.e., for any 
	$g \in C_c^\infty(T\mathbb{T}^n,\mathbb{R})$, we have 
	\begin{align*}
	 \int g\, d\widetilde{\mathfrak{M}}_1
	&= \langle \partial_1 g(x,\omega), \psi_1(x,\varepsilon) \rangle \\
	&\quad +\sum_{k\in\Z^n\backslash \{{\bf 0}\}} \frac{\langle \partial_2 g(x,\omega), \partial_{pp}H(0)k\rangle}{\langle k,\omega\rangle} f_k \, e^{i\langle k,x\rangle}\\
	&\quad +\langle \partial_2 g(x,\omega), \partial_{pp}H(0)c_1\rangle,
	\end{align*}
where $\widetilde{\mathfrak{M}}_1$, $\psi_1(x,\varepsilon)$, $c_1$ are as defined in \eqref{def:Mtilde}, \eqref{eq:psi-expansion}, \eqref{def:cvarepsilon}, and $f_k$ denotes the Fourier coefficient of $f$ for $k\in\mathbb{Z}^n$.
\end{thm}

\begin{proof}
By the KAM theorem (Lemma~\ref{lem:KAM for T^n flow}), the Hamilton--Jacobi equation \eqref{eq:HJCepsilon} admits a smooth solution $u_{\varepsilon}$ for every $\varepsilon \ge 0$. Moreover, there exists a diffeomorphism $\psi_\varepsilon : \mathbb{T}^n \to \mathbb{T}^n$ such that, for each $\varepsilon \ge 0$, the projected Hamiltonian flow $\pi \circ \varphi^t_{H_\varepsilon}$ is conjugated by $\psi_\varepsilon$ to the linear flow on $\mathbb{T}^n$ with frequency vector $\omega$.

Since the Mather measures are supported on Lipschitz graphs, we introduce the vector fields
\[
V_\varepsilon(x) := \partial_p H_\varepsilon\bigl(x, du_{\varepsilon}(x) + c_\varepsilon\bigr)
= \partial_p H\bigl(du_{\varepsilon}(x) + c_\varepsilon\bigr), \qquad x \in \mathbb{T}^n,
\]
and
\[
V_0(x) := \partial_p H(0)
= \omega, \qquad x \in \mathbb{T}^n.
\]
If we expand
\[
u_\varepsilon = u_0 + \varepsilon u_1 + o(\varepsilon), 
\qquad 
c_\varepsilon = \varepsilon c_1 + o(\varepsilon),
\]
then the vector field \(V_\varepsilon\) admits the asymptotic expansion
\[
V_\varepsilon(x) 
= \partial_p H\bigl( d u_\varepsilon(x) + c_\varepsilon \bigr)
= V_0(x) + \varepsilon \partial_{pp}H(0)(du_1+ c_1) + o(\varepsilon).
\]
Using equation \eqref{eq:HJCepsilon} together with Proposition~\ref{Prop:Ceplison}, we have
\[
\alpha(\varepsilon) = H\bigl(d u_{\varepsilon}+c_\varepsilon \bigr) + \varepsilon f(x)
= \alpha(0) + \varepsilon ([f]+\langle \omega, c_1 \rangle) + o(\varepsilon),
\]
which gives
\[
\langle \omega, du_1(x)\rangle 
= [f] - f(x) 
= \{f\}(x).
\]
Writing the Fourier expansion $u_1(x) = \sum_{k\in\mathbb{Z}^n} u_{1,k}\, e^{i\langle k,x\rangle}$, we obtain
\[
\langle \omega, d u_1(x)\rangle 
= \sum_{k\neq {0}} i \langle \omega, k\rangle\, u_{1,k}\, e^{i\langle k,x\rangle} 
= \sum_{k\neq {0}} f_k\, e^{i\langle k,x\rangle},
\]
which yields $u_{1,k} = \frac{f_k}{i\langle k,\omega\rangle}$ for $k\neq {0}$.
Consequently, we have
\begin{equation}\label{eq:V1k}
d u_1(x) = \sum_{k\neq {0}} i k \, u_{1,k} \, e^{i\langle k,x\rangle} 
= \sum_{k\neq {0}} \frac{k f_k}{\langle k, \omega \rangle} \, e^{i\langle k,x\rangle}.
\end{equation}
Due to the definition of the KAM torus and the definition of
\(\psi_\varepsilon\) in Lemma~\ref{lem:KAM for T^n flow}, it follows that
\[
\mathfrak{M}(\varepsilon) = (\psi_\varepsilon)_\# \mathfrak{M}(0).
\]
Then for any $g\in C_c^\infty(T\mathbb{T}^n,\mathbb{R})$ we have
\[
	\int g\,d\widetilde{\mathfrak{M}}(\varepsilon)
	- \int g\,d\widetilde{\mathfrak{M}}(0)
	= \int_{\mathbb{T}^n} \Bigl[ g(\psi_\varepsilon(x), V_\varepsilon(\psi_\varepsilon(x))) - g(x, V_0(x)) \Bigr]\, d\mathfrak{M}(0).
\]
By the regularity properties guaranteed by the KAM theorem, the maps $\psi_\varepsilon$ and the vector fields $V_\varepsilon$ admit asymptotic expansions in powers of $\varepsilon$:
\[
\psi_\varepsilon(x) = x + \varepsilon \psi_1(x,\varepsilon) + o(\varepsilon), \qquad
V_\varepsilon(x) = V_0(x) + \varepsilon \partial_{pp}H(0)(du_1+ c_1) + o(\varepsilon),
\]
so that
\[
V_\varepsilon(\psi_\varepsilon(x)) = V_0(x) +  \varepsilon \partial_{pp}H(0)(du_1+ c_1) + o(\varepsilon).
\]
Applying the first-order multivariate Taylor expansion of the smooth function
$g(x,p)$ around $(x, \omega)$ gives
\[
\begin{aligned}
	 & \; g(\psi_\varepsilon(x), V_\varepsilon(\psi_\varepsilon(x))) - g(x, \omega) \\
	=& \langle \partial_1 g(x,\omega), \psi_\varepsilon(x) - x \rangle
	+ \langle \partial_2 g(x,\omega), V_\varepsilon(\psi_\varepsilon(x)) - \omega \rangle
	+ o(\varepsilon) \\
	=& \varepsilon\, \langle \partial_1 g(x,\omega), \psi_1(x,\varepsilon) \rangle
	+ \varepsilon\, \langle \partial_2 g(x,\omega), \partial_{pp}H(0)(du_1+ c_1) \rangle
	+ o(\varepsilon).
\end{aligned}
\]
Substituting the expression for $d u_1(x)$ from \eqref{eq:V1k}, we obtain
\[
\begin{aligned}
	& \; g(\psi_\varepsilon(x), V_\varepsilon(\psi_\varepsilon(x))) - g(x, \omega) \\
	=& \; \varepsilon \, \langle \partial_1 g(x,\omega), \psi_1(x,\varepsilon) \rangle
	+ \varepsilon \sum_{k\neq {0}} 
	\frac{\langle \partial_2 g(x,\omega), \partial_{pp}H(0)k\rangle}{\langle k,\omega\rangle} 
	f_k \, e^{i\langle k,x\rangle} \\
	& \; + \varepsilon \, \langle \partial_2 g(x,\omega), \partial_{pp}H(0)c_1\rangle
	+ o(\varepsilon).
\end{aligned}
\]
Integrating over $\mathfrak{M}(0)$ and taking the limit as $\varepsilon \to 0$, we obtain the desired expression and establish the claimed linear response of $\widetilde{\mathfrak{M}}(\varepsilon)$. 
\end{proof}

\begin{cor}\label{cor:W1-Lip}
	Assume that \eqref{ass:A4} and \eqref{ass:A5} hold, and let 
	$f \in C^{\infty}(\mathbb{T}^n,\mathbb{R})$ be the smooth potential associated with the perturbed Hamiltonian \eqref{eq:pert-ham}. 
	Then, for $|\varepsilon| \ll 1$,
	\[
	W_1\bigl(\widetilde{\mathfrak{M}}(\varepsilon),\,\widetilde{\mathfrak{M}}(0)\bigr)
	\leq \mathcal{O}(\varepsilon).
	\]
	In other words, the Wasserstein distance between the perturbed and unperturbed Mather measures 
	is Lipschitz continuous with respect to the perturbation parameter~$\varepsilon$.
\end{cor}


\begin{thebibliography}{99}

\bibitem{Bal}
V. Baladi, Linear response, or else, in {\it Proceedings of the International Congress of Mathematicians, Seoul 2014}, Vol. III, 525--545, Kyung Moon Sa, Seoul, 2014.

\bibitem{Bangert99}
V. Bangert, Minimal measures and minimizing closed normal one-currents, Geom. Funct. Anal. {\bf 9} (1999), no.~3, 413--427. MR1708452.

\bibitem{BernardC11}
P. Bernard, Existence of $C^{1,1}$ critical sub-solutions of the Hamilton--Jacobi equation on compact manifolds, Ann. Sci. \'Ec. Norm. Sup\'er. (4) {\bf 40} (2007), no.~3, 445--452. MR2339280.

\bibitem{Bohr47}
H. Bohr, {\it Almost Periodic Functions}, Chelsea, New York, 1947. MR0020163.

\bibitem{Bolotin03openproblem}
S. Bolotin, New connections between dynamical systems and PDEs, Notes from the workshop, American Institute of Mathematics (AIMS), Palo Alto, July 2003. Available at: https://aimath.org/WWN/dynpde/dynpde.pdf.

\bibitem{Cannarsa_Sinestrari_book}
P. Cannarsa and C. Sinestrari, {\it Semiconcave Functions, Hamilton--Jacobi Equations, and Optimal Control}, Vol. 58 Progress in Nonlinear Differential Equations and Their Applications, Birkh\"auser Boston, Inc., Boston, MA, 2004.

\bibitem{Cassels1957}
J. W. S. Cassels, {\it An Introduction to Diophantine Approximation}, Cambridge University Press, Cambridge, 1957.

\bibitem{ChengActa2011}
C.-Q. Cheng, Non-existence of KAM torus, Acta Math. Sin. (Engl. Ser.) {\bf 27} (2011), no.~2, 397--404. MR2754043.

\bibitem{Con}
G. Contreras, Ground states are generically a periodic orbit, Invent. Math. {\bf 205} (2016), no.~2, 383--412.

\bibitem{Con2}
G. Contreras, Proof of the C2 Mane's conjecture on surfaces, ArXiv: 2408.01009, 2024.



\bibitem{CarandSorrentino2021}
C. Carminati, S. Marmi, D. Sauzin and A. Sorrentino,  On the regularity of Mather's $\beta$-function for standard-like twist maps, Adv. Math. {\bf 377} (2021), Paper No. 107460, 22 pp. MR4186003.

\bibitem{Davini06}
A. Davini and A. Siconolfi, A generalized dynamical approach to the large time behavior of solutions of Hamilton--Jacobi equations, SIAM J. Math. Anal. {\bf 38} (2006), no.~2, 478--502. MR2237158.

\bibitem{DaviniFathi2016}
A. Davini, A. Fathi, R. Iturriaga, and M. Zavidovique, Convergence of the solutions of the discounted Hamilton--Jacobi equation: convergence of the discounted solutions, Invent. Math. {\bf 206} (2016), no.~1, 29--55. MR3556524.

\bibitem{Fathi04}
A. Fathi and A. Siconolfi, Existence of $C^1$ critical subsolutions of the Hamilton--Jacobi equation, Invent. Math. {\bf 155} (2004), no.~2, 363--388. MR2031431.

\bibitem{Fathi05cv}
A. Fathi and A. Siconolfi, PDE aspects of Aubry--Mather theory for quasiconvex Hamiltonians, Calc. Var. Partial Differential Equations {\bf 22} (2005), no.~2, 185--228. MR2106767.

\bibitem{FathiSorrentino2009}
A. Fathi, A. Giuliani, and A. Sorrentino, Uniqueness of invariant Lagrangian graphs in a homology or a cohomology class, Ann. Sc. Norm. Super. Pisa Cl. Sci. (5) {\bf 8} (2009), no.~4, 659--680. MR2647908.

\bibitem{Fathiweakkam}
A. Fathi, {\it Weak KAM Theorem in Lagrangian Dynamics}, Cambridge University Press, 2025, to appear.

\bibitem{FR1}
A. Figalli and L. Rifford, Closing Aubry sets I, Comm. Pure Appl. Math. {\bf 68} (2015), no.~2, 210--285. MR3298663.

\bibitem{FR2}
A. Figalli and L. Rifford, Closing Aubry sets II, Comm. Pure Appl. Math. {\bf 68} (2015), no.~3, 345--412. MR3310519.

\bibitem{Gal}
S. Galatolo and J. Sedro, Quadratic response of random and deterministic dynamical systems, Chaos {\bf 30} (2020), no.~2, 023113.

\bibitem{Sorrentino21}
S. Galatolo and A. Sorrentino, Quantitative statistical stability and linear response for irrational rotations and diffeomorphisms of the circle, Discrete Contin. Dyn. Syst. {\bf 42} (2022), no.~2, 815--839. MR4369031.

\bibitem{ZhangCPDE}
B. Hu, S. N. T. Tu, and J. L. Zhang, Polynomial convergence rate for quasi-periodic homogenization of Hamilton--Jacobi equations and application to ergodic estimates, Comm. Partial Differential Equations {\bf 50} (2025), no.~1-2, 211--244. MR4858223.

\bibitem{Herman79}
M.-R. Herman, Sur la conjugaison diff\'erentiable des diff\'eomorphismes du cercle \`a des rotations, Inst. Hautes \'Etudes Sci. Publ. Math. No. 49 (1979), 5--233. MR0538680.

\bibitem{Klein2021ETDS}
S. Klein, X.-C. Liu, and A. Melo, Uniform convergence rate for Birkhoff means of certain uniquely ergodic toral maps, Ergodic Theory Dynam. Systems {\bf 41} (2021), no.~11, 3363--3388. MR4321715.

\bibitem{YanETDS2014}
Z. Liang, J. Yan, and Y. Yi, Viscous stability of quasi-periodic tori, Ergodic Theory Dynam. Systems {\bf 34} (2014), no.~1, 185--210. MR3163030.

\bibitem{Mane92}
R. Ma{\~n}{\'e}, On the minimizing measures of Lagrangian dynamical systems, Nonlinearity {\bf 5} (1992), 623--638.

\bibitem{Mane96}
R. Ma{\~n}{\'e}, Generic properties and problems of minimizing measures of Lagrangian systems, Nonlinearity {\bf 9} (1996), no.~2, 273--310. MR1384478.

\bibitem{Mather91}
J. N. Mather, Action minimizing invariant measures for positive definite Lagrangian systems, Math. Z. {\bf 207} (1991), no.~2, 169--207. MR1109661.

\bibitem{Natio96}
K. Naito, Fractal dimensions of almost periodic attractors, Ergodic Theory Dynam. Systems {\bf 16} (1996), no.~4, 791--803. MR1406434.

\bibitem{Po}
J. P\"oschel, Integrability of Hamiltonian systems on Cantor sets, Comm. Pure Appl. Math. {\bf 35} (1982), 653--695.


\bibitem{VillaniOT}
C. Villani, {\it Optimal Transport: Old and New}, Vol. 338 of Grundlehren der Mathematischen Wissenschaften, Springer-Verlag, Berlin, 2009.

\end{thebibliography}
\end{document}